\documentclass[11pt]{amsart}

\usepackage[OT2, T1]{fontenc}
\usepackage{url}
\usepackage{amsmath}
\usepackage{array}
\usepackage{graphicx}
\usepackage{amsfonts}
\usepackage{amssymb}
\usepackage{amstext}
\usepackage{amsthm}
\usepackage{hyperref}
\usepackage{colonequals}
\usepackage{enumitem}
\usepackage{cleveref}
\usepackage[all,cmtip]{xy}
\usepackage{fullpage}
\usepackage{mathrsfs}
\usepackage{tikz-cd}
\usetikzlibrary{calc,babel}
\numberwithin{equation}{subsection}

\newtheorem{theorem}{Theorem}
\numberwithin{theorem}{section}
\newtheorem{lemma}[theorem]{Lemma}
\newtheorem{proposition}[theorem]{Proposition}
\newtheorem{corollary}[theorem]{Corollary}

\newtheorem{conj}[theorem]{Conjecture}

\theoremstyle{definition}
\newtheorem{defn}[theorem]{Definition}

\theoremstyle{remark}
\newtheorem{rem}[theorem]{Remark}

\newtheorem{example}[theorem]{Example}

\tikzset{
  curve/.style={
    settings={#1},
    to path={
      (\tikztostart)
      .. controls ($(\tikztostart)!\pv{pos}!(\tikztotarget)!\pv{height}!270:(\tikztotarget)$)
      and ($(\tikztostart)!1-\pv{pos}!(\tikztotarget)!\pv{height}!270:(\tikztotarget)$)
      .. (\tikztotarget)\tikztonodes
    },
  },
  settings/.code={%
    \tikzset{quiver/.cd,#1}%
    \def\pv##1{\pgfkeysvalueof{/tikz/quiver/##1}}%
  },
  quiver/.cd,
  pos/.initial=0.35,
  height/.initial=0,
}

\newcommand{\MT}{{\mathrm{MT}}}

\newcommand{\diag}{{\mathrm{diag}}}

\DeclareMathOperator{\GL}{GL}

\DeclareMathOperator{\GSp}{GSp}

\DeclareMathOperator{\SO}{SO}
\DeclareMathOperator{\Or}{O}
\DeclareMathOperator{\GSpin}{GSpin}

\DeclareMathOperator{\Gal}{Gal}
\DeclareMathOperator{\End}{End}
\DeclareMathOperator{\Hom}{Hom}
\DeclareMathOperator{\Aut}{Aut}

\DeclareMathOperator{\Res}{Res}
\DeclareMathOperator{\Spec}{Spec}

\DeclareMathOperator{\Pic}{Pic}
\DeclareMathOperator{\NS}{NS}
\DeclareMathOperator{\tr}{tr}

\begin{document}

\title[Densities of Picard Rank Jumps for K3 Surfaces over Global Fields]{Densities of Picard Rank Jumps for K3 Surfaces over Global Fields}

\author{Robin Huang}

\begin{abstract}
    Let $X$ be an algebraic K3 surface defined over a global field $K$ of characteristic $0$ or $p \geq 3$, assumed ordinary in the case where $K$ is a global function field. We provide an exact formula for the density of primes $p$ of $K$ at which the geometric Neron-Severi group of the reduction $X_p$ has a given rank in terms of the Galois representation attached to $X$. As an extension, we prove that there are infinitely many reductions of $X$ with an element of the geometric Neron-Severi group representing a fixed square class under the intersection pairing and prove results about the density of such places. In particular, we show that if $X$ corresponds to a sufficiently generic moduli point, a density zero set of places of $K$ satisfy this condition. 
\end{abstract}

\maketitle
\tableofcontents

\section{Introduction}

Let $X$ be an algebraic K3 surface defined over a global field $K$. The geometric Neron-Severi group $\NS(\overline{X})$ is well-known to always be contained in that of its reduction $X_p$ modulo a good prime $p$ of $K$ via the specialization map
\[\mathrm{sp}: NS(\overline{X}) \hookrightarrow NS(\overline{X_p}). \]
When $K$ is a number field Charles in \cite{charles2014picard} studies the density of primes for which the above map is an isomorphism; more precisely, allowing finite extensions of the base field $K$, he states the difference in Picard rank for a density 1 set of primes in terms of the rank of the transcendental lattice and its ring of Hodge structure endomorphisms. One of our paper's main results provides an exact formula for the density of primes corresponding to each Picard rank difference for a general global field without passing to a field extension. 

\subsection{Explicit formula for density of Picard rank jumps}
\begin{theorem}\label{thm.main}
    Let $X$ be a K3 surface over a global field $K$ with geometric Picard rank $\rho(\bar{X})$. Assume that $X$ is ordinary and does not have CM when $K$ is a global function field. Let $V$ be the rational transcendental Hodge structure of $X$ when $K$ is a number field or that of the canonical lift of $X_{\bar{K}}$ when $K$ is a global function field. Let $F$ denote the endomorphism field of $V$ as a Hodge structure.\footnote{By a result of Zarhin, $F$ is totally real or CM, see \ref{subsec.mtnumfield}.}
    \begin{enumerate}
        \item If $F$ is CM and the dimension of $V$ as an $F$ vector space is even, then $\rho(\overline{X}) = \rho(\overline{X}_p)$ for a density 1 set of places $p$ of $K$. 
        \item  If $F$ is totally real and the dimension of $V$ as an $F$ vector space is odd, then $\rho(\overline{X}) + [F: \mathbb{Q}]= \rho(\overline{X}_p)$ for a density 1 set of places $p$ of $K$. 
        \item Otherwise, we may compute the density of primes of $p$ of $K$ such that $\rho(\overline{X}_p) = \rho(\overline{X}) + c$ for some $c \geq 0$ as
        \[\frac{a_c}{\# \pi_0}\]
        where $\pi_0$ is the component group of the $\ell$-adic monodromy group and $a_c$ is computable explicitly in terms of the Galois action, see \ref{sec.a}. 
    \end{enumerate}
\end{theorem}

\begin{rem}
In the function field case, the condition that $X$ does not have CM is to ensure that $X$ does not correspond to an isotrivial family of K3 surfaces. This assumption is equivalent to the Mumford-Tate group of $V$ not being a torus. Concretely, this requires the dimension of $V$ as an $F$ vector space being at least 2 when $F$ is CM and $\dim_F V_F \geq 3$ when $F$ is totally real (the latter condition is automatic, see \cite[Lemma 3.2]{van2008real}). 
\end{rem}

\begin{rem}
    The assumption that $X$ is ordinary when $K$ is a function field is used to gain some control on the $\ell$-adic monodromy group via an analog of the Mumford-Tate conjecture, see \ref{subsec.MTff} for a summary of the relevant result of \cite{jiang2023p}. A condition requiring maximal monodromy would yield a similar result beyond the ordinary case.
\end{rem}

\subsection{Previous works}
We briefly describe some related results and how Theorem \ref{thm.main} sits in relation to them. Most straightforwardly, one implication of the Tate conjecture (proved by \cite{nygaard1985tate, maulik2014supersingular,madapusi2015tate,charles2016birational,ito2021cm,kim20162,pera2020erratum}) is $\rho(\bar{X}_p)$ is always even. Consequently, if $X$ has odd geometric Picard rank, reduction modulo every prime yields a jump in Picard rank. 

Charles generalizes this dimension condition over a number field $K$ in \cite{charles2014picard}. In particular, he shows that if the dimension of the transcendental lattice as an $F$ vector space is odd and $F$ is totally real, for every prime $p$ of $K$ we have $\rho(\overline{X}_p) \geq \rho(\overline{X}) + [F:\mathbb{Q}]$ and up to an appropriate extension of $K$, this is an equality for a density 1 set of primes. 

On the other hand, when $\rho(\overline{X})$ is even, the Costa, Elsenhans and Jahnel in \cite{costa2020distribution} argue that primes $p$ whose corresponding Froebnii act by determinant $-1$ on the transcendental lattice give $\rho(\overline{X}_p) \geq \rho(\overline{X}) + 2$. To detect such primes, they introduce the so-called jump character. When $F$ is CM, using the Hodge conjecture Theorem 7.7 of \cite{elsenhans2024component} gives an explicit formula for the value of the jump character via the Galois action. Our proof recovers this formula without assumptions on $F$ or using the Hodge conjecture.\footnote{Inspired by the unpublished manuscript of Will Sawin \cite{sawin_distribution_nodate}.} More notably, Theorem \ref{thm.main}(3) yields a positive density of primes that give jumps in Picard rank not explained by any of the aforementioned phenomena. For more details on these points, see \ref{sec.a} and the examples in \ref{subsec.examples}.

\subsection{Cycles with a prescribed norm}

Theorem \ref{thm.main} answers how often a reduction of a given K3 surface carries an extra algebraic class. A natural generalization is to ask whether the new divisor class has norm in a prescribed square class, with the norm being induced by the intersection pairing. 

In the case where $A/K$ is an absolutely simple abelian variety, the condition of a reduction $A_v$ carrying an extra algebraic class with square norm corresponds to $A_v$ carrying additional idempotent endomorphisms. Thus, representing this particular square class corresponds to the abelian variety splitting into a product of smaller dimensional abelian varieties. Chavdarov in \cite[Theorem 2.2]{chavdarov1995generic} proved that for a family of curves over a finite field with full symplectic monodromy, a density one set of fibers have absolutely simple Jacobians. Zywina proved that after base extension, an abelian variety defined over a number field is isogenous to a product of lower dimensional abelian varieties for a density 1 set of primes if the endomorphism algebra is non-commutative; otherwise, the density is 0, see \cite[Theorem 1.2]{zywina2014splitting}.

In the case of a family mapping to a GSpin Shimura variety, the condition of representing a fixed square class translates to studying intersections of our family with special divisors of the form $Z(Dm^2)$, whose points correspond to Kuga-Satake abelian varieties with real multiplication by an order in $\mathbb{Q}(\sqrt{D})$, see \ref{subsec.gspin}. In \cite{maulik2022picard} and \cite{maulik2022reductions}, Maulik, Shankar and Tang prove that for a generically ordinary family of K3 surfaces over a proper curve $C$ in characteristic at least 5, there exist infinitely many points lying on certain families of special divisors $\bigcup Z(m)$ where $m$ takes values in a positive density subset of the natural numbers. 

We generalize their result to the family $\cup Z(Dm^2)$ for a fixed $D$, which is notably indexed by a density 0 subset of the natural numbers. This is the global function field analog of \cite{shankar2022exceptional}.

\begin{theorem}\label{thm.infinitude}
Let $\mathscr{S}$ denote the canonical integral model over $\mathbb{Z}_{(p)}$ of the GSpin Shimura variety associated to a $\mathbb{Z}$-lattice $(L, Q)$ of signature $(b, 2)$ with $b \geq 2$ such that $p$ does not divide the discriminant of $(L, Q)$. Let $C$ be an irreducible smooth proper curve with a finite morphism $C \rightarrow \mathscr{S}_{\mathbb{F}_p}$ such that the generic point of $C$ is ordinary and that the image of $C$ does not lie in any special divisors $Z(m)$. Then for each fixed $D > 0$ represented by $(L_{\mathbb{Q}},Q)$ there exist infinitely many $\bar{\mathbb{F}}_p$-points lying on $\bigcup_{m > 0} Z(Dm^2)$. 
\end{theorem}

\begin{rem}
    For $b \geq 3$, $L$ is a lattice of rank at least 5. Consequently, every square class $D > 0$ is represented, and the statement is true for any choice of $D$. 
\end{rem}

\begin{rem}
    The properness condition can be removed via an argument as in \cite{tayou2024picard}, which uses a toroidal compactification to bound the boundary contributions. 
\end{rem}

In a slightly different direction, Jiang, Shankar, and Yang prove in \cite{jiang2026picard} that with suitable large monodromy and non-exceptionality hypotheses, the Noether-Lefschetz locus is dense. This result does not restrict to a family of divisors corresponding to a specific square class. On the other hand, shortly before posting this present paper, we learned from the authors of the forthcoming paper \cite{jiang2026sintegral} that they proved a more general theorem, which includes our result as a special case.

On the other hand, we provide a density 0 result that applies to the generic Mumford-Tate condition $F = \mathbb{Q}$. 

\begin{theorem}\label{thm.sqclass}
    Let $X$ be as in Theorem \ref{thm.main} and suppose $F = \mathbb{Q}$. Then a density 0 set of places $p$ have an extra line bundle $L \in \NS(\overline{X_p})$ such that $Q(L) = Dm^2$ for a fixed $D$ and any $m$.  
\end{theorem}

\begin{rem}
    Some level of Mumford-Tate genericity is necessary. For instance, at a place $p$ of $K$ where the reduction $X_p$ satisfies
    \[\rho(\overline{X_{p}}) \geq \rho(\overline{X}) + 4  \]
    there are new divisor classes representing every single square class. Thus, in the case where Theorem \ref{thm.main} yields a positive density of places satisfying the above condition, a density 0 result is false, see Lemma \ref{lemma.meyer}. 
\end{rem}

\subsection{Proof strategies} 
We briefly describe the strategies in proving our theorems, starting with the strategy of Theorem \ref{thm.main}. Following \cite{charles2014picard}, we study the Galois representation 
\[\rho_{\ell}: \Gal_K \rightarrow \Or(H^2_{\text{ét}}(X, \mathbb{Q}_{\ell}(1)) \]
where $\Gal_K$ is the absolute Galois group of $K$. In particular, the Tate conjecture states that subspaces of the second cohomology group invariant under the Galois action are spanned by algebraic cycles. Thus studying geometric Picard rank of the reduction corresponds to studying the number of Frobenius eigenvalues that are equal to a root of unity. 

To study the Frobenius eigenvalues, it helps to have control on the \emph{$\ell$-adic monodromy group} $G_{\ell}$, an algebraic group defined by taking the Zariski closure of the image of $\rho_{\ell}$. Over a number field, a key input is the Mumford-Tate conjecture for K3 surfaces, proven by Tankeev in \cite{tankeev1990k3,tankeev1995surfaces}. It states that the neutral component of monodromy $G_{\ell}^{\circ}$ agrees with the $\mathbb{Q}_{\ell}$-Mumford-Tate group. 

To replace the above input in the function field case, we use Jiang's analog of the Mumford-Tate conjecture for a generically ordinary family of K3 surfaces in \cite{jiang2023p}. Another key novelty of our proof is to study the entire $\ell$-adic monodromy group rather than passing to a field extension $K^{conn}$ of $K$ such that the $\ell$-adic monodromy group becomes connected. To do so, we describe the components of the monodromy group, following the ideas of Sawin in \cite{sawin_ordinary_2015}, \cite{sawin_distribution_nodate}, and Cantoral-Farfan, Li, Mantovan, Pries, and Tang in \cite{cantoral_farfan_positive_2025}. In particular, we describe the component group as a subquotient of the normalizer of $G_{\ell}^{\circ}$, and describe this latter group. We provide a simpler description than the proof of \emph{loc. cit.}, which must first pass through a subquotient of a Weyl group. Using this description, we understand what a generic element of each component of the $\ell$-adic monodromy looks like, and then prove a Chebotarev density result on $\ell$-adic Lie groups to conclude. 

The above strategy is a bit too coarse for the proof of Theorem \ref{thm.sqclass}, as it does not distinguish the norms of each new line bundle. Instead, we use a sieve method similar to Chavdarov in \cite{chavdarov1995generic} and Zywina in \cite{zywina2014splitting}. See \ref{sec.four} for a more detailed discussion of our strategy.




\subsection{Organization of the paper.} In \S 2, we review some preliminaries on the algebraic monodromy group for a K3 surface defined over a global field and define GSpin Shimura varieties and their special divisors. In \S 3, we describe the components of the algebraic monodromy group, use this to prove Theorem \ref{thm.main}, and provide examples of our formula in use. In \S 4, we use the methods of Maulik, Shankar, and Tang to prove Theorem \ref{thm.infinitude}, concerning a density 0 set of special divisors indexed by a rational single square class. Finally, in \S 5, we prove Theorem \ref{thm.sqclass} and discuss the necessity of a Mumford-Tate genericity condition. 

\subsection{Acknowledgements and AI use statement} I am grateful to my advisor Yunqing Tang for introducing this problem for me and being a wellspring of advice and encouragement. I would also like to thank Fangu Chen, Ruofan Jiang, Wanlin Li, and Yuji Okitani for helpful discussions. This work was partially supported by the NSF GRF grant DGE 2146752 and the NSF grant DMS-2342225. 

Large language models suggested some ideas for the proofs of Lemma \ref{lemma.meyer} and \ref{lemma.lines}. While these lemmas are important for the proofs of our results, they are relatively simple, technical points and not constitute main ideas of our work. Additionally, all writing was done by the author.

\section{Preliminaries}

In this section, we review the necessary background for our paper, starting with the case where $K$ is a number field.

\subsection{K3 Surfaces and the Mumford-Tate conjecture over a number field}\label{subsec.mtnumfield}

Let $X$ be a K3 surface over a number field $K$. Fix an embedding $\iota: K \hookrightarrow \mathbb{C}$ and let $X_{\iota} = X \times_{K, \iota} \mathbb{C}$. We put $H^i_B(X, \mathbb{Q}) := H^i(X_{\iota}(\mathbb{C}), \mathbb{Q})$ and let $\NS(X) := \NS(X_{\iota})$ to be the geometric Neron Severi group of $X$. 

Consider the weight 2 Hodge structure on $H^2_B(X, \mathbb{Q})$ with polarization given by the cup product. We have the decomposition 
\[H^2_B(X, \mathbb{Q}) = \NS(X)_{\mathbb{Q}} \oplus V \]
where the rational transcendental Hodge structure $V$ is the orthogonal complement of $\NS(X)_{\mathbb{Q}}$ under the cup product. By the Lefschetz Theorem on $(1,1)$ classes, $V$ is the minimal sub-Hodge structure of $H^2(X, \mathbb{Q})$ such that 
\[H^{0,2}(X) \oplus H^{2,0}(X) = H^2(X_{\iota}, \mathcal{O}_X) \oplus H^0(X_{\iota}, \omega_X) \subseteq V \otimes \mathbb{C} \]
The Hodge structure $V$ has polarization 
\[\psi: V \otimes V \to \mathbb{Q}(-2) \]
induced by the cup product on $H^2(X, \mathbb{Q})$. 

By \cite[Theorems 1.5.1 and 1.6]{zarhin1983hodge}, the algebra of Hodge structure endomorphisms $\End_{HS}(V)$ is a totally real or CM field $F$, whose maximal totally real field is $F^+$ with degree $[F^+: \mathbb{Q}] = d$. The polarization induces an involution $e \mapsto e'$ on $F$, which is complex conjugation when $F$ is CM and the identity when $F$ is totally real. Let $V_F$ denote $V$ considered as an $F$ vector space, with $\dim_F V_F = n$. Thus, in the totally real case $\dim V = dn$ and in the CM case $\dim V = 2dn$. 

Let $\mathbb{S} := \Res_{\mathbb{C}/\mathbb{R}} \mathbb{G}_m$ denote the Deligne torus and recall that the data of a Hodge structure on a $\mathbb{Q}$ vector space $H$ is equivalent to the data of a representation 
\[\lambda: \mathbb{S} \rightarrow \GL(H_{\mathbb{R}}). \]
\begin{defn}
    We define \emph{the Mumford Tate group} of the rational Hodge structure $H$ to be the smallest algebraic subgroup $\MT(H) \subseteq \GL(H)$ such that $\MT(H)_\mathbb{R}$ contains $\lambda(\mathbb{S})$. 
\end{defn}

\begin{rem}\label{rem.absolutehodge}
    Since Hodge cycles on products of K3 surfaces are absolute Hodge (see \cite[Chapter I, 6.26]{deligne2009hodge}), our definitions of the Mumford-Tate group and the endomorphism field $F$ are independent of our choice of embedding $K \hookrightarrow \mathbb{C}$. 
\end{rem}

The Mumford-Tate Group of a polarized Hodge structure is a connected reductive group. Further, given any tensor construction
\[\bigoplus_{a,b \in \mathbb{Z}_{\geq 0}} H^a \otimes (H^{\vee})^b\]
with Hodge structure inherited from $H$, sub-Hodge structures are exactly the subspaces invariant under the action of $\MT(H)$ and $(0,0)$-Hodge tensors are exactly the invariants under the Mumford-Tate group. In particular, we conclude 
\[ (V \otimes V^{\vee})^{\MT(V)} = \End_{HS}(V) =  F \]
meaning $\MT(V)$ centralizes $F$ in the orthogonal similitude group. We may strengthen this observation using the following result of Zarhin. 

\begin{proposition}[{\cite[Theorem 2.2.1]{zarhin1983hodge}}]\label{prop.centralizer}
    The Mumford-Tate group of $V$ is the centralizer of $F$ in the group of orthogonal similitudes $\mathrm{GO}(V, \psi)$. 
\end{proposition}

Now we introduce the $\ell$-adic picture. Fix an algebraic closure $\bar{K}$ of $K$ together with an extension $\bar{\iota}: \bar{K} \hookrightarrow \mathbb{C}$ of $\iota$ and let $\overline{X} = X_{\bar{K}}$. Under the action of the absolute Galois group on $H^2_{\text{ét}}(\overline{X}, \mathbb{Q}_{\ell}(1))$, the $\mathbb{Q}_{\ell}$ span of the Neron Severi group $\NS(\bar{X})_{\mathbb{Q}_{\ell}}$ and its orthogonal complement $V_{\ell}(1)$ under the cup product are both globally invariant. As its notation suggests, by the Betti-étale comparison $V_{\ell}(1) \cong V \otimes \mathbb{Q}_{\ell}(1)$. By the Galois-equivariance of cup product and Poincaré duality in Galois cohomology, the polarization 
\[H^2(X,\mathbb{Q}_{\ell}) \otimes H^2(X,\mathbb{Q}_{\ell}) \rightarrow \mathbb{Q}_{\ell}(-2) \]
gives a Galois invariant bilinear form 
\begin{equation}\label{eq.galoiseq}
    \psi_{\ell}: V_{\ell}(1) \otimes V_{\ell}(1) \rightarrow \mathbb{Q}_{\ell}
\end{equation}
Thus, for a fixed prime $\ell$, we have the Galois representation
\[\rho_{\ell}: \Gal(\bar{K}/K) \rightarrow \mathrm{O}(V_{\ell}(1)). \]
Note that the Tate twist introduced makes $\rho_{\ell}$ land in the orthogonal group rather than the group of orthogonal similitudes. 
\begin{defn}
    The \emph{transcendental $\ell$-adic monodromy group} $G_{\ell}$ of $X$ is the Zariski closure of the image of $\rho_{\ell}$ in the orthogonal group $O(V_{\ell}(1),\psi_{\ell})$. 
\end{defn}
For a finite extension $K'/K$, since $G_{K'}$ is an open subgroup of finite index in $G_{K}$, whence the neutral component $G^{\circ}_{\ell}$ is invariant under taking finite base extensions. Further, the representations
\[\{\rho_{\ell}: \Gal(\bar{K}/K) \rightarrow \GL(V_{\ell}(1))  \}_{\ell \neq \mathrm{char}(K)} \]
form a compatible system in the sense that the characteristic polynomial of the $\rho_{\ell}$ image of a Frobenius element has $\mathbb{Q}$-coefficients and is independent of $\ell$. Indeed, the Weil conjectures imply this for the full cohomology group, and the Néron-Severi part is the scalar extension of a finite-image integral representation and is thus compatible. Additionally, these representations are semisimple using Faltings' semisimplicity theorem \cite{faltings1983endlichkeitssatze} and the Kuga-Satake construction (see \S 3 for a definition). Thus by \cite[133]{serre2013oeuvres} and \cite[Proposition 6.14]{larsen1992independence}, the group $\rho_{\ell}^{-1}(G_{\ell}^{\circ}(\mathbb{Q}_{\ell})) \subset \Gal(\bar{K}/K)$ is independent of the choice of $\ell$. This allows us to make the following definition. 
   
\begin{defn}
    Let $K^{conn}$ be the smallest finite extension of $K$ over which the transcendental $\ell$-adic monodromy group becomes connected. More precisely, let $K^{conn}$ be the subfield of $\bar{K}$ fixed by $\rho_{\ell}^{-1}(G_{\ell}^{\circ}(\mathbb{Q}_{\ell}))$. Thus, the component group $\pi_0 := \pi_0(G_{\ell})$ is isomorphic to $\Gal(K^{conn}/K)$.
\end{defn}

The above definitions hold in more general situations, and it is conjectured that for a smooth projective variety, the Mumford-Tate group is related to the algebraic monodromy group as follows. 

\begin{conj}[The Mumford-Tate Conjecture]
    Let $X$ be a smooth projective variety over a number field $K$ with a fixed embedding $\iota: K \hookrightarrow \mathbb{C}$. Let $G_{\ell}$ be the $\ell$-adic monodromy group attached to the étale cohomology group $H^i_{\text{ét}}(\overline{X}, \mathbb{Q}_{\ell})$. There exists a canonical isomorphism  
    \[G^{\circ}_{\ell} \cong \MT(H^i(X_{\iota}(\mathbb{C}), \mathbb{Q}))_{\mathbb{Q}_{\ell}} \]
    where $X_{\iota}$ is the base change of $X$ along $\iota$. 
\end{conj}

In our setting where $X$ is a K3 surface over a number field and $i=2$, the Mumford-Tate conjecture was proven by Tankeev. Note that by Remark \ref{rem.absolutehodge}, the choice of embedding $K \hookrightarrow \mathbb{C}$ does not affect the statement. 

\begin{theorem}[Tankeev, {\cite{tankeev1990k3}, \cite{tankeev1995surfaces}}]\label{thm.MT}
    The Mumford-Tate conjecture is true for the second cohomology group of a K3 surface defined over a number field. Consequently, it also holds for the transcendental part $V$, meaning under the Betti-étale comparison isomorphism 
    $V(1) \otimes \mathbb{Q}_{\ell} \cong V_{\ell}(1)$ we have 
    \[G_{\ell}^{\circ} \cong \MT(V(1))_{\mathbb{Q}_{\ell}} = C_{\mathrm{SO}(V(1),\psi)}(F).\]
\end{theorem}

\begin{proof}
    The statement for the full second cohomology group was proven by Tankeev. 
    
    For the latter statement, note that we have the decomposition of the Tate twisted Hodge structures
    \[H^2_B(X,\mathbb{Q})(1) = \NS(X)(1) \oplus V(1) \]
    where the Mumford-Tate group acts trivially on the first summand because $\NS(X)$ is spanned by Hodge classes. Therefore the Mumford-Tate groups of $H^2_B(X,\mathbb{Q})(1)$ and $V(1)$ are equal. 
    
    On the $\ell$-adic side, note that the Galois group acts trivially on $\NS(\bar{X})_{\mathbb{Q}_{\ell}}$ after a finite extension of $K$. Since the neutral component of monodromy is insensitive to base extensions, the Mumford-Tate conjecture holds for $V(1)$. The final equality is by Proposition \ref{prop.centralizer}, Galois equivariance of \ref{eq.galoiseq}, and the fact that $\mathrm{SO}$ is the neutral connected component of the orthogonal group. 
\end{proof}

We may describe $G^{\circ}_{\ell}$ more explicitly in cases depending on whether $F$ is totally real or CM. Following \cite{zarhin1983hodge}, there exists a unique form $\phi: V \otimes V \to F$ satisfying $\tr_{F/\mathbb{Q}}(a\phi(x,y)) = \psi(ax,y)$. In particular, if $F$ is CM $\phi$ is a $F$-sesquilinear Hermitian form, and if $F$ is totally real, $\phi$ is symmetric $F$-bilinear. Thus, we have 
\begin{equation}\label{eq.neutralcomp}
    G^{\circ}_{\ell} = \begin{cases}
        \Res_{F/\mathbb{Q}}\SO(V_F, \phi)_{\mathbb{Q}_{\ell}} & F \text{ is totally real}\\
        \Res_{F^+/\mathbb{Q}}U(V_F, \phi)_{\mathbb{Q}_{\ell}} & F \text{ is CM}
    \end{cases}
\end{equation}

\subsection{GSpin Shimura varieties and the Kuga-Satake construction}\label{subsec.gspin}

To describe an analog to the Mumford-Tate conjecture in the case of a global function field, we first review some basic definitions and properties of GSpin Shimura varieties and the Kuga-Satake construction following \cite{pera2016integral} and \cite{andreatta2018faltings}. Along the way, we define special divisors on these Shimura varieties corresponding to K3 surfaces with additional algebraic classes in the second cohomology group.

We first note for a K3 surface over $\mathbb{C}$ the transcendental part of the second cohomology group equipped with the negative cup-product is a quadratic lattice of signature $(20-\rho,2)$ by the Hodge Index theorem. This motivates the construction below.

\subsubsection{Canonical integral models of GSpin Shimura varieties}

Let $(L,Q)$ be a quadratic $\mathbb{Z}_{(p)}$-lattice of signature $(b,2)$, $b \geq 1$, and assume that $L$ is self-dual. For a commutative $\mathbb{Z}_{(p)}$-algebra $R$, let $L_R := L \otimes R$ and let $C(L)$ denote the Clifford algebra of $(L,Q)$ with $\mathbb{Z}/2\mathbb{Z}$-grading 
\[C(L) = C^+(L) \oplus C^-(L) \]
into its even and odd parts respectively. We recall the definition of the spinor similitude group $G := \GSpin(L)$ of $(L, Q)$ as the reductive algebraic group over $\mathbb{Z}_{(p)}$ with the functor of points
\[\GSpin(L)(R) = \{g \in C^+(L_R)^{\times} : g (L_R) g^{-1} = L_R \} \]
for any $\mathbb{Z}_{(p)}$-algebra $R$. The group $G$ acts on $V$ via $g \cdot v = gvg^{-1}$ for $v \in L_R$ and is the extension of the special orthogonal group by $\mathbb{G}_m$. 
\[1 \rightarrow \mathbb{G}_m \rightarrow G \xrightarrow{g \mapsto g \cdot } SO(L) \rightarrow 1 \]
Letting $[-,-]$ denote the bilinear form $[x,y] = Q(x+y) - Q(x) - Q(y)$, the group $G(\mathbb{R})$ acts on the Hermitian symmetric domain 
\[D_L =  \{z \in L_{\mathbb{C}} : [z,z] = 0, [z, \bar{z}] < 0 \}/\mathbb{C}^{\times}\]
via the action of $G_{\mathbb{Q}} \rightarrow \SO(L_{\mathbb{Q}})$. Following \cite{andreatta2018faltings}, the pair $(G_{\mathbb{Q}}, D_L)$ defines a Shimura datum with reflex $\mathbb{Q}$. Indeed, for $[z] \in D_L$ arising from $z \in L_{\mathbb{C}}$, let 
\[h_{[z]}: \Res_{\mathbb{C}/\mathbb{R}} \mathbb{G}_m \rightarrow G_{\mathbb{R}}\]
be the unique morphism of algebraic groups over $\mathbb{R}$ that induces the Hodge decomposition on $L_{\mathbb{C}}$ given by 
\[L_{\mathbb{C}}^{1,-1} = \mathbb{C}z, \quad L_{\mathbb{C}}^{0,0} = (\mathbb{C}z \oplus \mathbb{C}\bar{z})^{\perp}, \quad  L_{\mathbb{C}}^{-1,1} = \mathbb{C}\bar{z}. \]
Now let $\mathbb{K} \subset G(\mathbb{A}_f)$ be a compact open subgroup contained in $G(\mathbb{A}_f) \cap C(L \otimes \hat{\mathbb{Z}})^{\times}$, and assume that $\mathbb{K}_p = G(\mathbb{Z}_p)$. Then we have the GSpin Shimura variety $\mathcal{S}_{\mathbb{K}} := Sh(G_{\mathbb{Q}}. D_L)_{\mathbb{K}}$ over $\mathbb{Q}$. By \cite[Theorem 2.3.8]{kisin2010integral}, $\mathcal{S}_{\mathbb{K}}$ admits a canonical smooth integral model $\mathscr{S}_{\mathbb{K}}$ over $\mathbb{Z}_{(p)}$. We will drop the subscript when the level structure is clear from context.

\subsubsection{The Kuga-Satake construction}

We now relate GSpin Shimura varieties to Siegel Shimura varieties parametrizing abelian varieties. Let $H$ denote $C(L)$ with the action of right multiplication by itself. The representation $G \rightarrow C(L)$ induced by left multiplication together with a suitable symplectic form on $H$ yields an embedding of the GSpin Shimura datum into a Siegel Shimura datum
\[(\GSpin(L_{\mathbb{Q}}), D_L) \rightarrow (\GSp(H_{\mathbb{Q}}), \mathfrak{H}^{\pm}). \]
This embedding is called the \emph{Kuga-Satake embedding}. From this embedding, there exists a \emph{Kuga-Satake abelian scheme} $A^{KS} \rightarrow \mathcal{S}$ whose first integral Betti cohomology $\mathbb{H}_B$ is the local system induced by the $G_{\mathbb{Q}}$ action on $H_{\mathbb{Q}}$ (for a description of this local system, see for instance \cite[Section 3]{pera2016integral}). Roughly speaking, the Kuga-Satake abelian scheme is the pullback of the universal abelian scheme over the Siegel Shimura variety along the Kuga-Satake embedding. In fact, this Kuga-Satake abelian scheme extends to an abelian scheme $\mathcal{A}^{KS} \rightarrow \mathscr{S}$, equipped with a left Clifford action. Let $\mathbb{H}_{dR}, \mathbb{H}_{\ell,\text{ét}}$ denote the first relative de Rham cohomology and integral $\ell$-adic cohomology of $\mathcal{A}^{KS} \rightarrow \mathscr{S}$ for $\ell \neq p$. Let $\mathbb{H}_{cris}$ denote the first relative crystalline cohomology of $\mathcal{A}_{\mathbb{F}_p}^{KS} \rightarrow \mathscr{S}_{\mathbb{F}_p}$. 

The left action of $L$ on $H$ via left multiplication induces a $G_{\mathbb{Q}}$-equivariant map 
\[L_{\mathbb{Q}} \rightarrow \End_{C(L)}(H \otimes \mathbb{Q}) \]
and thus there is a local system $(\mathbb{L}_{B}, \mathbb{L}_{dR}, \mathbb{L}_{\ell, \text{et}}, \mathbb{L}_{cris})$ with an embedding 
\begin{equation}\label{eq.realizations}
    \mathbb{L}_{?} \hookrightarrow \End_{C(L)}(\mathbb{H}_{?})
\end{equation}
for each $? \in \{B, dR, \{\ell, \text{ét}\}, cris\}$ compatible with the various comparison isomorphisms. Further, the local system defined above is equipped with a natural quadratic form $Q$ such that $f \circ f = Q(f) \text{Id}$ for a section $f$ of $\mathbb{L}_{?}$.

\subsubsection{Special endomorphisms and divisors}

Let $T$ be a $\mathscr{S}$ scheme and let $A_T^{KS}$ to be the pull-back of the Kuga-Satake abelian scheme to $T$.  

\begin{defn}[{\cite[Definition 4.3.1]{andreatta2018faltings}}]
    For a $\mathscr{S}$ scheme $T$, an endomorphism 
    \[v \in \End_{C(L)}(A_T^{KS})\] 
    is \emph{special} if all of its homological realizations land in the images of the embeddings of \ref{eq.realizations}. Let $V(A_T^{KS})$ to be the space of special endomorphisms.
\end{defn}

\begin{rem}
    By Proposition 4.3.4 of \emph{loc. cit.} we may check if an endomorphism is special at a geometric point. More precisely, an endomorphism $v \in \End_{C(L)}(A^{KS}_T)$ is special if and only if there exists a geometric point $t \in T$ such that $v_t \in \End_{C(L)}(A_{t}^{KS})$ is special. 
\end{rem}

\begin{lemma}[{\cite[Lemma 5.12]{pera2016integral}}]
    The quadratic form $Q$ on $V(A_T)$ given by $f \circ f = Q(f) \text{Id}$ is positive definite. 
\end{lemma}

\begin{defn}
    For $m \in \mathbb{Z}_{>0}$ the \emph{special divisor} $\mathcal{Z}(m)$ is the Deligne-Mumford stack over $\mathscr{S}$ with functor of points 
    \[\mathcal{Z}(m)(T) = \{v \in V(A_T^{KS}) | Q(v) = m \} \]
    for any $\mathscr{S}$ scheme $T$. By an abuse of notation, let $\mathcal{Z}(m)$ also denote the image of $\mathcal{Z}(m)$ in $\mathscr{S}$. Then by \cite[Proposition 4.5.8]{andreatta2018faltings}, $\mathcal{Z}(m)$ is an effective Cartier divisor and is flat over $\mathbb{Z}_{(p)}$. Thus, $Z(m) = \mathcal{Z}(m)_{\mathbb{F}_p}$ is still an effective Cartier divisor. 
\end{defn}

Special endomorphisms correspond to line bundles orthogonal to an ample class on the associated K3 surface. More precisely, by \cite[Theorem 4.17(4)]{madapusi2015tate} for a geometric point $\bar{x} \rightarrow \mathscr{S}$ we have an isomorphism 
\[V(A_{\bar{x}}^{KS}) \cong \langle \xi_{\bar{x}} \rangle^{\perp} \subseteq \Pic(X_{\bar{x}}) \]
where $\xi_{\bar{x}}$ is the specialization of a polarization class to $\bar{x}$. In particular, this isomorphism is compatible with the quadratic forms and all cohomological realizations. Restricting to the associated transcendental local system identifies special endomorphisms with new line bundles of $X_{\bar{x}}$. Thus, we use $Q$ to also denote the quadratic form induced by the cup product. 

\subsection{Mumford-Tate conjecture for K3 surfaces over a global function field}\label{subsec.MTff}

Now we summarize an analogue of the Mumford-Tate conjecture over a global function field by Jiang in \cite{jiang2023p}. 

Let $X$ be an ordinary polarized K3 surface over a global function field $K$. Let $C \rightarrow \mathbb{F}_q$ be the smooth projective curve with function field $K$. After replacing $C$ with a suitable nonempty open subscheme, there exists a smooth, projective, polarized family of K3 surfaces
\[\pi: \mathcal{X} \rightarrow U \]
such that the generic fiber $\mathcal{X}_{\eta} \cong X$. To apply Jiang's construction, we describe a map of this family into a GSpin Shimura variety using facts about the Kuga-Satake period map in \cite{madapusi2015tate}.  

\begin{proposition}\label{prop.periodmap}
    After replacing $U$ by a nonempty open subscheme and then by a finite étale cover $U'$, there exist an even quadratic lattice $\widetilde{L}$, self-dual at $p$, and a morphism    
    \[f: U' \xrightarrow{\iota^{KS}} \mathscr{S}(L_{d})_{\mathbb{F}_q}^{ord} \rightarrow \mathscr{S}(\tilde{L})_{\mathbb{F}_q}^{ord} \]
    compatible with cohomological realizations. 
\end{proposition}

\begin{proof}
Replacing $U$ with a finite étale cover where $\mathcal{X}$ has a primitive polarization if necessary, this family determines a classifying morphism to the moduli stack $M_{2d}$ of primitively polarized K3 surfaces of degree $2d$ for some $d > 0$ defined over $\mathbb{Z}[1/2]$. Again possibly taking a finite etale cover of $U$ for the appropriate level structure, this classifying morphism extends to the cover $\tilde{M}_{2d}$ of $M_{2d}$. Then by Proposition 4.7 \emph{loc. cit.}, there exists a morphism 
\[\iota^{KS}: U' \rightarrow \mathscr{S}(L_d)_{\mathbb{F}_q}. \]
Here $L_d$ is the primitive lattice orthogonal to the class of the primitive polariazation. Choose a self-dual enlargement of $L_d \subset \tilde{L}$ as in \cite[Lemma 6.8]{pera2016integral}. After choosing appropriate level structure, there exists a morphism of integral models 
\[\mathscr{S}(L_d) \rightarrow \mathscr{S}(\tilde{L}) \]
After composing morphisms, we obtain the morphism in the statement. By \cite[Propositions 6.5 and 7.15]{pera2016integral}, we may deduce that the morphisms are compatible with cohomological realizations. 
\end{proof}

We first make some establish some notations and definitions. As above, let $V_{\ell}(1)$ denote the transcendental part of $H^2_{\text{ét}}(X, \mathbb{Q}_{\ell}(1))$, defined to be the orthogonal complement under cup product of the rational Neron-Severi group. Similar to the number field case, this is invariant under the Galois action.

\begin{defn}
    The transcendental $\ell$-adic monodromy group $G_{\ell}$ of $X$ is the Zariski closure of the image of 
    \[\rho_{\ell}: \pi_1^{\text{ét}}(U,\bar{x}) \rightarrow \Or(V_{\ell}(1),\psi)\]
\end{defn}

\begin{rem}
    At a geometric generic point $\bar{\eta} = \Spec \bar{K}$ of $U$ we have a similar representation of $\Gal(\bar{K}/K)$, which actually factors through $\rho_{\ell}$. We will refer to these representations interchangeably.
\end{rem}

For the same reason as in the number field case, there is a unique smallest extension of $U$ such that the monodromy group becomes connected. Let $K^{conn}$ refer to the corresponding extension of $K$. 

Now, compose $f$ with the Kuga-Satake embedding and by an abuse of notation, call this resulting map $f$
\[f: U' \rightarrow \mathscr{S}(\tilde{L})_{\mathbb{F}_q}^{ord} \rightarrow \mathscr{A}_{g, \mathbb{F}_q}.\]
In particular, $\mathscr{S}(\tilde{L})$ is Jiang's so-called naive integral model. 

\begin{defn}
    Define $\MT(f)$ to be the generic Mumford-Tate group of a minimal special subvariety of $\mathscr{A}_{g,\bar{\mathbb{Q}}}$ whose Zariski closure in $\mathscr{A}_{g, \bar{\mathbb{Z}}}$ contains the image of $f$. Let $\mathrm{Hg}(f)$ be the associated Hodge group, fitting into the exact sequence 
    \[1 \rightarrow \mathbb{G}_m \rightarrow \MT(f) \rightarrow \mathrm{Hg}(f) \rightarrow 1. \]
\end{defn}

\begin{rem}
    More concretely, we may define $\MT(f)$ to be the Mumford-Tate group of the canonical lift of a geometric generic point of $U'$. This definition is equivalent by \cite[Lemma 4.3]{jiang2023p} and the fact that the Mumford-Tate group for a Hodge-structure is determined by its endomorphisms. This characterization is what makes our definition of $F$ a valid one in Theorem \ref{thm.main}.  
\end{rem}

\begin{theorem}[{\cite[Theorem 1.3]{jiang2023p}}]\label{thm.FFMT}
    Let $X/K$, $\mathcal{X} \rightarrow U'$, and $f$ be as above. Then
    \[\mathrm{Hg}(f)_{\mathbb{Q}_{\ell}} \cong G^{\circ}_{\ell}. \]
    In particular, the $G_{\ell}^{\circ}$ is of the same form as \ref{eq.neutralcomp}. 
\end{theorem}

\begin{proof}
    We apply Jiang's theorem and apply the adjoint action to the relevant representations. The second claim is by Zarhin's characterization in \cite{zarhin1983hodge} and the fact that the neutral component is insensitive to finite étale extensions of the base. 
\end{proof}

\section{Density of Picard Rank Jumps}

In this section, we prove Theorem \ref{thm.main} of our paper. The following result will serve as our main framework in our density computations. For $g \in G_{\ell}(\overline{\mathbb{Q}_{\ell}})$, define  
\begin{equation}\label{eq.mult}
m(g) := \sum_{\zeta \in \mu_{\infty}} \mathrm{mult}_{\zeta}(\det((T-g)|V_{\ell}))
\end{equation}
be the function that counts the eigenvalues of $g$ equal to roots of unity with multiplicity. 
\begin{proposition}\label{prop.Chebotarev}
    For any integer $c \geq 0$, the density of primes $p$ of $K$ such that $\rho(\overline{X}_p) = \rho(\overline{X}) + c$ is equal to 
    \[\frac{\# \{\gamma \in \pi_0(G_{\ell}) : \min\{m(g) : g \in \gamma G_{\ell}^{\circ}(\overline{\mathbb{Q}_{\ell}})  \} = c \} }{\# (\pi_0(G_{\ell}))} \]
\end{proposition}

\begin{proof}
    By \cite{madapusi2015tate} and \cite{kim20162}, $X_p$ satisfies the Tate conjecture. In particular, this combined with the nondegeneracy of the pairing on Neron-Severi implies that Frobenius invariant subspaces of our transcendental lattice $V$ correspond to elements in $\NS(\overline{X}_p) \setminus \NS(\overline{X})$. 
    
    Now, let $\mathcal{G}_{\ell} := \overline{\mathrm{im}(\rho_{\ell})}^{Zar} \subseteq O(V, \psi)_{\mathbb{Z}_{\ell}}$, which we call the $\mathbb{Z}_{\ell}$-monodromy group. Let $\Gamma$ be the image of $\rho_{\ell}$ in the $\ell$-adic Lie group $\mathcal{G}_{\ell}(\mathbb{Z}_{\ell})$. As a closed subgroup of a compact set, we may equip $\Gamma$ with the Haar measure of total mass 1. Note that $\Gamma$ is Zariski dense in $\mathcal{G}_{\ell}$ by definition. 

    Given an element $\gamma$ of $\Gal(K^{conn}/K)$ in the number field case and of $\pi_1(U)/\pi_1(U_{K^{conn}})$ in the function field case, let $\Gamma^{(\gamma)}$ be the image of $\gamma \Gal(\bar{K}/K^{conn})$ and $\gamma \pi_1(U_{K^{conn}})$ in $\mathcal{G}_{\ell}(\mathbb{Z}_{\ell})$ respectively. Again we note that by definition, $\Gamma^{(\gamma)}$ is Zariski dense in the corresponding component $\mathscr{G}_{\ell}^{(\gamma)}$ of $\mathscr{G}_{\ell}$.

    For any integer $k \geq 0$, let $Z_k$ be the set of elements in $\mathcal{G}_\ell$ whose characteristic polynomial acting on $V$ has at least $k$ roots of unity as roots (counted with multiplicity). Because the characteristic polynomial has a fixed degree and there are only finitely many roots of unity of a given degree, each $Z_k$ is a conjugacy-invariant, closed subscheme of $\mathcal{G}_\ell$. 
    
    Since $Z_k \cap \Gamma$ is a closed analytic subset of $\Gamma$, its boundary measure is 0 by \cite[Proposition 5.9]{serre2016lectures}. Hence, Chebotarev density theorem (specifically, \cite[Theorem 9.13A]{rosen2013number} in the function field case) implies that the density of the set of primes for which $\rho_{\ell}(\mathrm{Frob}_p)$ is in $Z_c$ is equal to the Haar measure of $Z_c \cap \Gamma$ \cite[Corollary 6.10]{serre2016lectures}.
    
    Now, by \cite[5.2.1.2 and Proposition 5.12]{serre2016lectures}, the Haar measure of $Z_c \cap \Gamma$ is exactly the number of components of $G_{\ell}$ contained in $Z_c$ divided by the total number of connected components. Note that if the minimum number of roots of unity on a given component is at least $c$, then said component is by definition contained in $Z_c$. On the other hand, if the minimum number of roots of unity on this component is strictly less than $c$, then $Z_{c} \cap \mathcal{G}_\ell^{(\gamma)}$ is a proper closed subscheme of $\mathcal{G}_\ell^{(\gamma)}$. Since $Z_{c+1} \subseteq Z_{c}$, the density of primes yielding exactly $c$ roots of unity is the difference of these densities, which by additivity of the Haar measure equals the measure of $(Z_{c} \setminus Z_{c+1}) \cap \Gamma$. Thus, we deduce the density of primes $p$ such that $\rho(\overline{X}_p) = \rho(\overline{X}) +c $ is equal to the proportion of connected components of $\mathcal{G}_\ell$ on which the minimum number of eigenvalues equal to a root of unity (counted with multiplicity) for any element is exactly $c$.
\end{proof}

\subsection{Describing the component group}

In light of Proposition \ref{prop.Chebotarev}, our next goal is a group theoretic description of the component group $\pi_0 := G_{\ell}/G_{\ell}^{\circ}$ so that we may better understand the elements of each component. 

Let $V_{\ell} = V \otimes \mathbb{Q}_{\ell}$ and $\psi_{\ell} = \psi \otimes \mathbb{Q}_{\ell}$. From here on, we take $\ell$ to be totally split in $F/\mathbb{Q}$. Under this assumption we have that there is a decomposition of $V_{\ell}$ into $F^+$ eigenspaces
    \[V_{\ell} = \bigoplus_{\tau \in \Hom(F^+, \mathbb{Q}_{\ell})} V_{\ell}^{\tau} \]
so that $\psi_{\ell}$ decomposes into a sum of symmetric bilinear forms $\psi_{\ell}^{\tau}$ on each subspace $V_{\ell}$. In the totally real case, let the basis of $V_{\ell}^{\tau}$ be 
\begin{equation}\label{eq.TRbases}
    B_{\tau} = \{v_{\tau,1}, \dots, v_{\tau, n} \}.
\end{equation}
Without loss of generality, we may assume that we have chosen our basis and so that $\psi$ is the identity. Indeed, after picking a diagonal basis we may further choose $\ell$ to split completely in a finite Galois extension of $\mathbb{Q}$ containing the square roots of the diagonal coefficients. We may then choose bases $B_{\tau}$ such that the every $\psi_{\ell}^{\tau}$ has Gram matrix $I$.

For the CM case, given an embedding $\tau_i \in \Hom(F^+, \mathbb{Q}_{\ell})$, let $\tau_{i,1}, \tau_{i,2}: F \hookrightarrow \mathbb{Q}_{\ell}$ be the conjugate embeddings extending $\tau_i$. Each $V_{\ell}^{\tau_i}$ can be further decomposed into $V_{\ell}^{\tau_{i,1}}$ and $V_{\ell}^{\tau_{i,2}}$ with bases 
\begin{equation}\label{eq.CMbases}
    B_{\tau_{i,1}} = \{v_{\tau_i,1}, v_{\tau_i,2}, \dots v_{\tau_i,n} \} \qquad B_{\tau_{i,2}} = \{w_{\tau_2,1}, w_{\tau_2,2}, \dots, w_{\tau_2,n} \}
\end{equation}
such that $\psi_{\ell}^{\tau_i}$ is of the form
\[\begin{bmatrix}
    0 & I \\
    I & 0
\end{bmatrix} \]
To see why, let $\phi_i$ be the unique Hermitian form on $V_{\ell}^{\tau_i}$ viewed as a $F \otimes \mathbb{Q}_{\ell}$-module. With a suitable choice of $\ell$, we may write $\phi_i$ as the identity matrix under some basis. Let $B_{\tau_{i,j}}$ be the projections of this basis to $V_{\ell}^{\tau_{i,j}}$. Finally, let $A_{\tau} = B_{\tau_1} \cup B_{\tau_2}$ be the corresponding basis of $V_{\ell}^{\tau_i}$. 

Since $G^{\circ}_{\ell}$ commutes with the $F \otimes \mathbb{Q}_{\ell}$ action, using the bases of \ref{eq.TRbases} and \ref{eq.CMbases} we may write the elements of ${G}^{\circ}_{\ell}$ in a specific block diagonal form. When $F$ is totally real, in order to preserve the pairing on each $V_{\ell}^{\tau_i}$, our elements must be of the form 
\begin{equation}\label{eq.TRblocks}
    g = \diag(M_{\tau_1}, \dots, M_{\tau_d})
\end{equation}
where $M_{\tau_i} \in \SO(V_{\ell}^{\tau},\psi_{\ell}^{\tau})$. In the CM case, they are of the form 
\begin{equation}\label{eq.CMblocks}
    g = \diag(M_{\tau_{1,1}}, M_{\tau_{1,2}}, \dots, M_{\tau_{d,1}}, M_{\tau_{d,2}}) 
\end{equation}
where $M_{\tau_{i,j}} \in \GL(V_{\ell}^{\tau_{i,j}})$ and $M_{{\tau_{i,2}}} = M_{\tau_{i,1}}^* := (M_{\tau_{i,1}}^T)^{-1}$ in order to preserve the pairings on the $F^+ \otimes \mathbb{Q}_{\ell}$ eigenspaces. To summarize, we have 

\begin{lemma}\label{lemma.blocks}
    There exist isomorphisms 
    \[{G}^{\circ}_{\ell} \cong \begin{cases} 
        \prod_{i=1}^d \SO_{n, \mathbb{Q}_{\ell}} & \text{ if } $F$ \text{ is totally real }\\
        \prod_{i=1}^d \GL_{n,\mathbb{Q}_{\ell}} & \text{ if } $F$ \text{ is CM }.
    \end{cases}\]
    In the totally real case, the isomorphism is the identity. In the CM case, the isomorphism is given by $(M_1, M_1^*, \dots, M_d, M_d^*) \mapsto (M_1, \dots, M_d)$.
\end{lemma}

A key observation in describing our component group is that because $G_{\ell}$ normalizes $G^{\circ}_{\ell}$, we have $G_{\ell} \subseteq N_O(G_{\ell}^{\circ})$, with the latter being the normalizer of $G_{\ell}^{\circ}$ in the full orthogonal group. Quotienting by $G_{\ell}^{\circ}$, we have the injection 
\begin{equation}\label{eq.normalizer}
    \pi_0(G_{\ell}) \hookrightarrow N_{\Or}(G_{\ell}^{\circ})/G_{\ell}^{\circ}.
\end{equation}
We now focus on computing $N_{\Or}(G_{\ell}^{\circ})/G_{\ell}^{\circ}$ in each case.

\begin{proposition}\label{prop.normalizer}
    The group $N_{\Or}(G_{\ell}^{\circ})/G_{\ell}^{\circ}$ is isomorphic to $(\mathbb{Z}/2\mathbb{Z})^d \rtimes S_d$. Under this isomorphism, an element $((\delta_1, \dots, \delta_d), \sigma) \in (\mathbb{Z}/2\mathbb{Z})^d \rtimes S_d$ has as a coset representative
    \[b_{\sigma} \, \diag(r^{\delta_1}, \dots, r^{\delta_d}) \]
    where $b_{\sigma}$ is the block permutation matrix corresponding to $\sigma$, acting on the set of $F^+ \otimes \mathbb{Q}_{\ell}$-eigenspaces of $V_{\ell}$ and $r$ may be described as follows. 
    \begin{enumerate}
        \item When $F$ is totally real, $r$ is a matrix in $\Or(V_{\ell}^{\tau_i}, \psi_{\ell}^{\tau_i})$ with determinant $-1$. 
        \item When $F$ is CM, $r$ acts by swapping pairs of $F \otimes \mathbb{Q}_{\ell}$-eigenspaces via the matrix 
        \[\begin{bmatrix}
            0 & I_n \\
            I_n & 0 
        \end{bmatrix} \in \Or(V_{\ell}^{\tau_i}, \psi_{\ell}^{\tau_i}) \]
    \end{enumerate}
\end{proposition}

\begin{proof}
    In all cases, since $G_{\ell}$ normalizes $G_{\ell}^{\circ}$, it normalizes its centralizer $F \otimes \mathbb{Q}_{\ell}$. Thus, $x \in G_{\ell}$ acts as a $\mathbb{Q}_{\ell}$-algebra automorphism on $F \otimes \mathbb{Q}_{\ell}$, and permutes the $F \otimes \mathbb{Q}_{\ell}$-eigenspaces of $V_{\ell}$. 

    In the totally real case, consider $x \in N_{\Or}(G_{\ell}^{\circ})$. Then by the above description, $x$ sends $V_{\ell}^{\tau_i}$ to $V_{\ell}^{\tau_{\sigma(i)}}$ for some $\sigma \in S_d$. Letting $b_{\sigma}$ be the block permutation matrix representing $\sigma$, note that $y = b_{\sigma}^{-1} x$ fixes each $F \otimes \mathbb{Q}_{\ell}$-eigenspace and that by definition, we may write
    \[y = \diag(q_{\tau_1}, q_{\tau_2}, \dots, q_{\tau_d})\] 
    for $q_{\tau_i} \in \Or(V_{\ell}^{\tau_i}, \psi_{\ell}^{\tau_i})$. Then fix a diagonal reflection 
    \[r = \diag(-1, 1, \dots, 1) \]
    and let $\delta_i \in \{0,1\}$ satisfy $\det(q_{\tau_i}) = (-1)^{\delta_i}$. Then letting 
    \[g_{\tau_i} = r^{\delta_i} q_{\tau_i} \in \SO(V_{\ell}^{\tau_i}, \psi_{\ell}^{\tau_i})\]
    we have the equality 
    \[x = b_{\sigma} \, \diag(r^{\delta_1}, \dots, r^{\delta_d}) \, \diag(g_{\tau_i}, \dots, g_{\tau_d}) \]
    and thus taking the quotient by $G^{\circ}_{\ell}$, we note that each coset may be represented by $((\delta_1, \dots, \delta_d), \sigma) \in (\mathbb{Z}/2\mathbb{Z})^d \rtimes S_d$. For the reverse inclusion, by definition every block permutation and block reflection is orthogonal and normalizes $G^{\circ}_{\ell}$, with only the identity element landing in $G_{\ell}^{\circ}$. 

    When $F$ is CM, let $\dagger$ denote the adjoint involution induced by $\psi_{\ell}$, which restricts to complex conjugation on $F \otimes \mathbb{Q}_{\ell}$. Since an element $x \in N_O(G^{\circ}_{\ell})$ is orthogonal, we note that 
    \[(xax^{-1})^{\dagger} = xa^{\dagger}x^{-1} \]
    and thus the automorphism on $F \otimes \mathbb{Q}_{\ell}$ induced by $x$ commutes with complex conjugation. Thus, $x$ must permute the $F^+ \otimes \mathbb{Q}_{\ell}$-eigenspaces of $V_{\ell}$, of which there are $d$. Independently of said permutation, $x$ may preserve or exchange the corresponding $F \otimes \mathbb{Q}_{\ell}$-eigenspaces. This data naturally gives us the wreath product $(\mathbb{Z}/2\mathbb{Z})^d \rtimes S_d$. Suppose the action of $x$ corresponds to an element $(\epsilon, \sigma) \in (\mathbb{Z}/2\mathbb{Z})^d \rtimes S_d$ and let $b_{\epsilon, \sigma}$ denote this matrix. Then write
    \[y = b_{\epsilon,\sigma}^{-1}x, \]
    which preserves each $F \otimes \mathbb{Q}_{\ell}$ eigenspace of $V_{\ell}$. In the basis of \ref{eq.CMbases}, it may be written as 
    \[y = \diag(A_1, B_1, \dots, A_d, B_d) \]
    and preserving the form $\psi_{\ell}^{\tau_i}$ in the chosen basis implies $(B_i^T)^{-1} = A_i$, so $y \in G^{\circ}_{\ell}$. Consequently, every coset of $N_O(G^{\circ})_{\ell}/G_{\ell}^{\circ}$ can be represented by a unique $b_{\epsilon,\sigma}$ corresponding to $(\epsilon, \sigma) \in (\mathbb{Z}/2\mathbb{Z})^d \rtimes S_d$. 
\end{proof}

\subsection{Picard Rank Jump Detection} In this section, we explain how to compute the constant $a_c$ corresponding to a Picard rank jump of size $c \geq 0$ as in Theorem \ref{thm.main}(3) and prove said theorem. 

\subsubsection{Definition of $a_c$}\label{sec.a}
For 
\[b\in\pi_0(G_\ell),\]
write its image under the isomorphism of Proposition \ref{prop.normalizer} as 
\[
(\epsilon, \sigma) \in (\mathbb{Z}/2\mathbb{Z})^d \rtimes S_d.
\]
Let $\mathcal{C}(\sigma)$ denote the set of orbits of
$\sigma$ acting on $\{1,\ldots,d\}$, including
singleton orbits. For \(C\in\mathcal{C}(\sigma)\), define the \emph{flip parity} of $C$ as 
\[ \epsilon(C)
    :=
    \sum_{i\in C}\epsilon_i
    \in\mathbb{Z}/2\mathbb{Z}.
\]
Then for $b$ as above, let  
\[J_C(b)
    :=
    \begin{cases}
        0, & \text{if } \epsilon(C) = 0
    \in\mathbb{Z}/2\mathbb{Z} ,\\
        2|C|, & \text{if } \epsilon(C) = 1 \in \mathbb{Z}/2\mathbb{Z}.
    \end{cases}
\]
Finally, we define 
\[
    J(b)
    :=
    \sum_{C\in\mathcal{C}(\sigma)}J_C(b)=
    2\sum_{\substack{C\in\mathcal{C}(\sigma)\\
                     \epsilon(C)=1}}|C|.
\]

Equivalently, embed
\[
    (\mathbb{Z}/2\mathbb{Z})^d\rtimes S_d
    \hookrightarrow
    S_{\{\pm1,\ldots,\pm d\}}
\]
in the natural way. Using standard facts about the wreath product, for $C\in\mathcal{C}(\sigma)$, the signed set
\[
    C\cup(-C)
\]
is the orbit of two cycles of length $|C|$ when $\epsilon(C)=0$, and is that of a single cycle of length $2|C|$
containing both $i$ and $-i$ when $\epsilon(C)=1$.
Consequently, $J(b)$ is the sum of the lengths of the cycles
in the signed permutation corresponding to $b$ that contain both $i$ and
$-i$ for some $i$.

In terms of the action on the set of $F \otimes \mathbb{Q}_{\ell}$-eigenspaces of $V_{\ell}$, the action of $\epsilon_i = 1 \in \mathbb{Z}/2\mathbb{Z}$ on $V_{\ell}^{\tau_i}$ is by reflection. $J(b)$ is the twice the total number of $F \otimes \mathbb{Q}_\ell$-eigenspaces that belong to an orbit on which $b$ acts with an odd number of reflections.

For $c\geq0$, define
\[
    a_c
    :=
    \#\left\{
        b\in\pi_0(G_\ell):J(b)=c
      \right\}.
\]

\begin{proof}[Proof of Theorem \ref{thm.main}]
    From Proposition \ref{prop.Chebotarev}, we would like to study the minimum number of eigenvalues of a given component. To do so, we will start with an element of the neutral component $\mathcal{G}_{\ell}^{\circ}$ and study the translates of it under the action of the component group as described in \ref{eq.normalizer} and Proposition \ref{prop.normalizer}.

    We start with the simplest case, where $F$ is CM and $n = \dim_F V_F$ is even. We construct an element $g \in {G}^{\circ}_{\ell}(\overline{\mathbb{Q}_{\ell}})$ with no roots of unity as eigenvalues and all of its translates satisfying the same property. Using the description of \ref{eq.CMblocks}, let $g$ be such that 
    \[M_{\tau_{i,1}} = \diag(R_{1}, \dots, R_{n/2}) \]
    where each $R_{j}$ is a $2 \times 2$ rotation of infinite order and the eigenvalues of $R_j$ are distinct. Then by Proposition \ref{prop.normalizer}, an element $b \in \pi_0$ acts on the embeddings $\{\tau_{1,1}, \tau_{1,2}, \dots, \tau_{d,1}, \tau_{d,2}\}$ by an element of $(\mathbb{Z}/2\mathbb{Z})^d \rtimes S_d$. Then for any $k \geq 0$, we write $b^k(\tau_{i,j})$ as the $k$-fold iterate of this action. Letting $m$ be the order of $b$, we note that $(bg)^m$ is again block diagonal. Indeed,
    \[(bg)^m = \diag\left(\prod_{j=0}^{m-1} M_{b^j(\tau_{1,1})}, \prod_{j=0}^{m-1} M_{b^j(\tau_{1,2})}, \dots, \prod_{j=0}^{m-1} M_{\,b^j(\tau_{d,1})}, \prod_{j=0}^{m-1} M_{b^j(\tau_{d,2})}\right)\]
    Since by our choice of $g$ we have $(M_{\tau_{i,j}}^{T})^{-1} = M_{\tau_{i,j}}$, we note that $(bg)^m$ is a block diagonal matrix where each block consists of a finite product of infinite order rotations. In particular, since $(bg)^{m}$ has no eigenvalues that are equal to a root of unity, we conclude $bg$ also has none. Applying Proposition \ref{prop.Chebotarev}, this shows (1). 

    Next let $F$ be totally real and $n$ is odd. We want to show that in each component of $\mathcal{G}_{\ell}$, every element has at minimum $d = [F:\mathbb{Q}]$ eigenvalues equal to a root of unity. We start with an element $g \in \mathcal{G}_{\ell}^{\circ}$, which by Lemma \ref{lemma.blocks} may be written in the form of
    \[g = \diag(M_{\tau_1}, \dots, M_{\tau_d}) \]
    with each block in the special orthogonal group of $V_{\ell}^{\tau_i}$ as in \ref{eq.TRblocks}. Note that since $\dim_F V_F$ is odd, each element of $\SO(V_{\ell}^{\tau_i}, \psi_{\ell}^{\tau_i})$ has at least one eigenvalue equal to 1. Let $M_{\tau_i}$ be a matrix consisting of $2 \times 2$ infinite order rotations with mutually disjoint spectra as above, with the final entry being $1$. As above, consider the image of $b \in \pi_0(G_{\ell})$ in $(\mathbb{Z}/2\mathbb{Z})^d \rtimes S_d$ and for convenience, set the reflection $r$ as in Proposition \ref{prop.normalizer} to be 
    \[r = \diag(1,1, \dots, -1).\]
    An argument as above shows that $(bg)^m$ is block diagonal with each block lying in $\SO(V_{\ell}^{\tau_i}, \psi_{\ell}^{\tau_i})$, meaning $(bg)^m$ has at least $d = [F:\mathbb{Q}]$ eigenvalues that are roots of unity. Our choice of $bg$ attains this minimum. This shows (2). 
    
    Now we prove (3). Suppose first that $F$ is totally real and that $n$ is even. As above, write 
    \[g = \diag(M_{\tau_1}, \dots, M_{\tau_d}). \]
    Then for $b \in \pi_0$, consider its image $(\gamma,\sigma) \in H$, and write $\gamma = (\gamma_1, \dots, \gamma_d)$ and $\sigma = \prod_{i} \sigma_i$ as in \ref{sec.a}. Let $m$ be the order of $\sigma$. As above, $(bg)^m$ is a block diagonal matrix which we analyze block by block. For $k \in \mathrm{supp}(\sigma_i)$, the $k$-th block is of the form
    \begin{equation}\label{eq.gamma}
        \left( \prod_{j=0}^{m_i-1} \gamma_{\sigma_i^j(k)} M_{\sigma_i^j(\tau_k)} \right)^{m/m_i},
    \end{equation}
    where $m_i$ is the order of $\sigma_i$ and $\gamma_i$ acts by a reflection matrix on $V_{\ell}^{\tau_i}$ when $\gamma_i = 1 \in \mathbb{Z}/2\mathbb{Z}$ and by the identity otherwise. We note that before exponentiating, the above product inside the parentheses is orthogonal and has determinant $-1$ if and only if $\sum_{k \in \operatorname{supp}(\sigma_i)} \gamma_k = 1$ in $\mathbb{Z}/2\mathbb{Z}$. Since $n$ is even, such a matrix must have at least one eigenvalue equal to $-1$ and at least one eigenvalue equal to $1$. Upon raising this product to the power of $m/m_i$, these eigenvalues map are still $\pm 1$, meaning the resulting exponentiated block retains at least two roots of unity as eigenvalues. Since there are $\operatorname{length}(\sigma_i)$ such blocks, we conclude that when $\sum_{k \in \operatorname{supp}(\sigma_i)} \gamma_k = 1$ in $\mathbb{Z}/2\mathbb{Z}$, the corresponding blocks have in total at least $2(\operatorname{length}(\sigma_i))$ eigenvalues equal to a root of unity. In each of these cases, we can see the minimum value is achieved by setting 
    \[M_{\tau_{i,1}} = \diag(R_{1}, \dots, R_{n/2}) \]
    as in the proof for (1) and let $\gamma_i$ act by reflecting the first basis element of each $V_{\ell}^{\tau_i}$.

    When $F$ is CM and $\dim_F V_F = n$ is odd, with the block $\diag(M_{\tau_{i,1}}, M_{\tau_{i,2}})$
    playing the role of $M_{\tau_i}$ in the argument for the totally real case, the idea is similar. Since $n$ is odd, the action of swapping the two $F \otimes \mathbb{Q}_\ell$ sub-eigenspaces over a given $F^+ \otimes \mathbb{Q}_\ell$ eigenspace can be represented by a block permutation matrix with determinant $(-1)^n = -1$. This swap thus acts as a reflection on the corresponding $F^+ \otimes \mathbb{Q}_\ell$ eigenspace, and the conclusion follows identically to the totally real case, noting that the minimum number of roots of unity is again achieved using a choice of $2 \times 2$ rotation matrices.
\end{proof}

\subsection{Comparison with the jump character}

The jump character of \cite{costa2020distribution} records the
determinant of the Galois action on the transcendental part of the second cohomology group. Since $G_{X,\ell}^{\circ}$ is contained in the special orthogonal
group, this character factors through the component group:
\[
    \tau_X\colon
    \pi_0(G_{X,\ell})
    \longrightarrow
    \{\pm1\},
    \qquad
    b\longmapsto\det(b).
\]
On the other hand, $J(b)$ yields the exact number of
root-of-unity eigenvalues forced by the component containing $b$.
In particular, it is possible to have
\[
    J(b)>0
    \qquad\text{but}\qquad
    \det(b)=1.
\]
Such a component corresponds to Picard rank jumps that are not detected by $\tau_X$. 

In the totally real case with $\dim_F V$ even, this is
group-theoretically possible if the image of the component group inside $(\mathbb{Z}/2\mathbb{Z})^d \rtimes S_d$ contains an element having an even total number of reflections but
an odd number of reflections around at least one cycle.  

In the CM case, we provide specific examples in which there are Picard rank jumps invisible to the jump character.

\subsection{Examples}\label{subsec.examples}

In the number field case, the following lemma from Will Sawin simplifies computations. 

\begin{lemma}[\cite{sawin_distribution_nodate}]
    Let $G$ be the motivic Galois group of the second cohomology group of a K3 surface $X$ defined over a number field $K$. Then for some embedding $F \hookrightarrow \bar{K}$, $KF$ is Galois over $K$ and the image of $G$ in $\Aut(F)$ is $\Gal(KF/K)$. 
\end{lemma}

In particular, since the base change of $G$ to $\mathbb{Q}_{\ell}$ is $G_{\ell}$, this identifies the Galois action of $F$ with the action on $V_{\ell}$. Thus, when $\ell$ is split, our description of the permutation on the $F \otimes \mathbb{Q}_{\ell}$-eigenspaces of $V_{\ell}$ agrees with the action of $\Gal(KF/K)$ on the embeddings $F \hookrightarrow \mathbb Q_\ell$. In particular, our proof above recovers the formula in \cite[Theorem 7.7(b)]{elsenhans2024component}. We also remark that in the case where our K3 surface has CM, as below, this lemma is a consequence of the main theorem of complex multiplication. 

\begin{example}\label{ex.jacobian}
    Let $C$ be the genus $2$ curve given by $y^2 = x^5 - 1$ and $X = \mathrm{Kum}(\mathrm{Jac}(C))$ be the Kummer surface associated to the Jacobian of $C$. Then in Example 8.2 of \cite{elsenhans2024component}, they show 
    \[\rho(\bar{X}) = 18, \quad F = \mathbb{Q}(\zeta_5) \]
    and the jump character $\tau_X = \left(\frac{5}{\cdot} \right)$. With our notation, $d = 2$ and $\dim_F V_F = 1$. Note that the CM type of $Jac(C)$ is primitive, and is thus not geometrically isogenous product of elliptic curves. 
    
    Elsenhans-Jahnel also show $\pi_0(G_{X,\ell}) \cong \Gal(F/\mathbb{Q})$ (Theorem 7.7 of \emph{\frenchspacing{loc. cit.}}). The injection $\pi_0(G_{X,\ell}) \rightarrow (\mathbb{Z}/2\mathbb{Z})^2 \rtimes S_2$ is given by the natural action of $\Gal(F/\mathbb{Q})$ on the embeddings of $F = \mathbb{Q}(\zeta_5)$. In particular, writing $\gamma$ as the generator of $\Gal(F/\mathbb{Q}) \cong \mathbb{Z}/4\mathbb{Z}$ and labelling the conjugate pairs of embeddings as 
    \[P_1 = \{\tau_1, \tau_4\}, \quad P_2 = \{\tau_2, \tau_3\}, \quad \tau_a(\zeta_5) = \zeta_5^a \]
    we have 
    \begin{align*}
        1 &\mapsto ((0,0),1),\\
        \gamma &\mapsto ((0,1),(12)),\\
        \gamma^2 &\mapsto ((1,1), 1),\\
        \gamma^3 &\mapsto ((1,0),(12))
    \end{align*}
    We thus have the following table 
    \begin{center}
        \begin{tabular}{|c|c|c|c|}
            \hline
            Frobenius class & $J(b)$ & $\det(b)$ & $p \mod 5$\\
            \hline
            $1$ & $0$ & $1$ & $1$\\
            \hline
            $\gamma$ & $4$ & $-1$ & $2$\\
            \hline 
            $\gamma^2$ & $4$ & $1$ & $4$\\
            \hline
            $\gamma^3$ & $4$ & $-1$ & $3$\\
            \hline
        \end{tabular}
    \end{center}
    Notably, $J(b_{\gamma^2}) = 4$ whereas $\tau_X(\gamma^2) = 1$. These primes satisfying $p \equiv 4 \mod 5$ have natural density $1/4$, contributing to a total of density $3/4$ of primes satisfying $\rho(\overline{X}_p) = 22$. 
\end{example}
The following example has the jump character completely trivial.  
\begin{example}\label{ex.product}
    Consider the Kummer surface 
    \[X = \mathrm{Kum}(E_1 \times E_2) \]
    of the product of elliptic curves  
    \[E_1: y^2 = x^3 + x \qquad  E_2: y^2 = x^3 + 4x^2 + 2x  \]
    in Example 8.3 of \cite{elsenhans2024component}, Elsenhans and Jahnel prove that 
    \[\rho(\bar{X}) = 18, \quad F = \mathbb{Q}(\sqrt{2}, i) \]
    and thus $\dim_F V_F = 1$. Notably, they also show that the jump character is trivial on $X$.

    The injection $\pi_0(G_{X,\ell}) \rightarrow (\mathbb{Z}/2\mathbb{Z})^2 \rtimes S_2 \cong D_8$ is given by the natural action of $\Gal(F/\mathbb{Q})$ on the embeddings of $F = \mathbb{Q}(\sqrt{2},i)$. In particular, let $\Gal(F/\mathbb{Q}) = \{1, c, s, cs\}$ where
    \begin{align*}
        c(i) = -i, \quad c(\sqrt{2}) = \sqrt{2}\\
        s(i) = i, \quad s(\sqrt{2}) = -\sqrt{2}
    \end{align*}
    Then the elements of $\pi_0$ in $D_8$ are given by 
    \[\{((0,0),1), ((1,1),1), ((0,0),(12)), ((1,1), (12)) \}. \]
    Most notably, the element $c$ corresponds to $b_c = ((1,1),1)$. From our definition in \ref{sec.a}, we have $J(b_c) = 4$ whereas $\det(b_c) = 1$. Thus, the natural density of rational primes $p$ such that $\rho(\overline{X}_p) = 22$ is $1/4$. Concretely, they correspond to primes $p$ such that both elliptic curves are supersingular, which is equivalent to $p \equiv 7 \mod 8$. 
\end{example}

\subsubsection{Splittings of abelian varieties}

Our result may also be applied to a class of abelian varieties and the unitary Shimura varieties that parametrize them. 

Let $F$ be a CM field of degree $2d$ with maximal totally real subfield $F^+$. Let $\Sigma = \{\tau_1, \dots, \tau_d\}$ be a CM type consisting of embeddings of $F$ into $\overline{\mathbb{Q}}$, with complex conjugates $\bar{\tau}_j$. Fix an integer $r \geq 1$. For any $a \in F$, consider the polynomial 
$$ P_a(T) = \prod_{j=1}^d (T - \tau_j(a))^{p_j} (T - \bar{\tau}_j(a))^{q_j} \in \overline{\mathbb{Q}}[T] $$
where $(p_1, q_1) = (r, 1)$ and $(p_j, q_j) = (r+1, 0)$ for $j > 1$. Let $E \subset \overline{\mathbb{Q}}$ be the field generated by the coefficients of $P_a(T)$ for all $a \in F$. Thus, for any $a \in \mathcal{O}_F$, we have $P_a(T) \in \mathcal{O}_E[T]$.

Consider the moduli problem that associates to a locally Noetherian $\mathcal{O}_E$-scheme $S$ the groupoid of triples $(B, i, \lambda)$, where $B \rightarrow S$ is an abelian scheme of relative dimension $d(r+1)$, $i: \mathcal{O}_F \hookrightarrow \mathrm{End}_S(B)$ is a ring homomorphism, and $\lambda: B \rightarrow B^{\vee}$ is a principal polarization, such that:
\begin{enumerate}
    \item $i(a)^{\dagger} = i(\bar{a})$, where $\dagger$ is the Rosati involution induced by $\lambda$ and $\bar{a}$ is complex conjugation on $F$.
    \item The action of $i(a)$ on $\mathrm{Lie}(B)$ has characteristic polynomial $\varphi_*(P_a(T))$, where $\varphi_*: \mathcal{O}_E[T] \rightarrow \mathcal{O}_S[T]$ is the map induced by the structure morphism $S \rightarrow \mathrm{Spec}(\mathcal{O}_E)$.
\end{enumerate}

This functor is represented by a Deligne--Mumford stack $\mathcal{M}_{F, \Sigma}(r,1)$ over $\mathcal{O}_E$, which is a disjoint union of unitary Shimura varieties. Similarly, we may define the moduli space $\mathcal{M}_{F,\Sigma}(1,0)$, where the signature condition requires the action of $\mathcal{O}_F$ on the Lie algebra to have signature $(1,0)$ at every embedding $\tau_j \in \Sigma$. Equivalently, for any $a \in \mathcal{O}_F$, the characteristic polynomial of its action on the Lie algebra is exactly $\prod_{j=1}^d (T - \tau_j(a))$.

Then, as in the construction of \cite[9.3]{shankar2022exceptional}, there exists a morphism of Shimura varieties corresponding to maps between motives from the GSpin Shimura variety in \S 2 and an orthogonal Shimura variety. We may thus obtain a similar result in this case. However, notably we may need to pass to the field over which the CM point is defined. 

\section{Infinitude of Reductions with Cycle of Prescribed Norm}

In this section we combine the methods of \cite{maulik2022picard} and \cite{maulik2022reductions} to prove Theorem \ref{thm.infinitude}. 

\subsection{Setup and Strategy}
Let $(L,Q)$ be a quadratic lattice over $\mathbb{Z}$ of signature $(b,2)$ with $b \geq 2$ such that $p \nmid \mathrm{disc}(L)$, and let $\mathscr{S}$ denote the canonical integral model of the GSpin Shimura variety associated to $(L,Q)$. Let $\mathcal{X} \rightarrow C$ be a family of K3 surfaces defined over a proper curve $C$ of characteristic $p \geq 5$ and suppose there exists a finite morphism $C \rightarrow \mathscr{S}_{\bar{\mathbb{F}}_p}$ such that the generic point of $\mathcal{X}$ maps to the ordinary locus. For a point $P \in C(\bar{\mathbb{F}}_p)$, consider the formal neighborhood $\hat{C}_P = \bar{\mathbb{F}}_p[[t]]$. Write $A := A_{\hat{C}_P}^{KS}$ for the pullback of the Kuga-Satake abelian variety over $\mathscr{S}_{\bar{\mathbb{F}}_p}$. We define $L_n$ to be the lattice of special endomorphisms of $A \mod t^n$, where $A \mod t^n$ fits into the following Cartesian diagram
\[\begin{tikzcd}
	{A \text{ mod } t^n} & A & {A^{KS}_{\bar{\mathbb{F}}_p}} \\
	{\Spec \bar{\mathbb{F}}_p[t]/t^n} & {\hat{C}_P} & {\mathscr{S}_{\bar{\mathbb{F}}_p}}
	\arrow[from=1-1, to=1-2]
	\arrow[from=1-1, to=2-1]
	\arrow[from=1-2, to=1-3]
	\arrow[from=1-2, to=2-2]
	\arrow[from=1-3, to=2-3]
	\arrow[from=2-1, to=2-2]
	\arrow[from=2-2, to=2-3]
\end{tikzcd}\]
Then it follows from the moduli-theoretic definition that the local intersection number $i_P(C.Z(m))$ between $C$ and $Z(m)$ at $P$ can be computed as
\[i_P(C.Z(m)) = \sum_{n=1}^{\infty} \# \{v \in L_n : Q(v) = m \}. \]

We now define our indexing set of special divisors. Let $D$ be a square class represented by $L$; for $b \geq 3$, this imposes no restriction on $D$. Then let 
\[S_{D}= \{Dx^2: x \in \mathbb{Z}_{>0}, (x,2p(\mathrm{disc}L)) = 1 \}\]
and  
\[S_{D,M} = \{m \in S_{D}: M \leq m \leq 2M \}. \]
Notably, $S_{D}$ is a density 0 set in the natural numbers. As in prior works, our strategy is to compare the sum $\sum_{m \in S_{D,M}} C.Z(m)$ with the corresponding sums of local intersection numbers. At a supersingular point $P$, we show that the sum 
\[\sum_{m \in S_{D,M}} i_P(C.Z(m)) = o\left(\sum_{m \in S_{D,M}} C.Z(m)\right) \]
as $M \to \infty$. On the other hand, we show the contribution of non-supersingular points asymptotically satisfies
\[\sum_{P \text{ non-s.s.}}\sum_{m \in S_{D,M}} i_P(C.Z(m)) = \alpha \sum_{m \in S_{D,M}} C.Z(m) \]
for some $0 < \alpha < 1$.

\subsection{Global Intersection}

We use Borcherds theory to control the asymptotic of the global intersection number by relating it to the coefficients of a vector valued modular-form. For a more detailed summary, see \cite[4.1.4]{maulik2022picard}.

Let $L^{\vee}$ denote the dual of $L$ in $L_{\mathbb{Q}}$ induced by $Q$ and let $\{\mathfrak{e}_{\mu \in L^{\vee}/L} \}$ denote the standard basis of $\mathbb{C}[L^{\vee}/L]$. By the modularity theorem of Borcherds \cite{borcherds1999gross} and \cite{howard2017arithmetic} in the arithmetic case, we have 
\[-(\omega.C) + \sum_{m=1}^{\infty} Z(m).C \]
is the $\mathfrak{e}_0$-component of a vector-valued form, where $\omega$ is the line bundle of modular forms on $\mathscr{S}_{\bar{\mathbb{F}}_p}$. More specifically, the Eisenstein part is equal to the $\mathfrak{e}_0$ component of $-(\omega.C)E_0$, where $E_0$ is a certain vector-valued Eisenstein form. Let $q_L(m)$ denote the coefficient of $q^m \mathfrak{e}_0$ in the $q$-expansion of $E_0$.

\begin{proposition}\label{prop.global}
    For all $m$ representable by $(L,Q)$ and $b = 2$ we have $m^{1-\epsilon} \ll_{\epsilon,L} |q_L(m)| \ll m^{1 + \epsilon}$. For $b \geq 3$ we have  $|q_L(m)| \asymp m^{b/2}$. Thus
    \[Z(m).C = |q_L(m)|(\omega.C) + \begin{cases}
    O_{\epsilon}(m^{1/2 + \epsilon}) & b = 2\\
    O(m^{(b+2)/4}) & b \geq 3
    \end{cases} .\]
    In particular, $\sum_{m \in S_{D,M}} Z(m).C \asymp (\omega.C) \sum_{m \in S_{D,M}} |q_L(m)| \asymp M^{(b+1)/2}$. 
\end{proposition}

\begin{proof}
    For our choices of $m$ above, the asymptotic on $|q_L(m)|$ for $b \geq 3$ is by \cite[Theorem 7.4]{maulik2022reductions}, whereas for $b = 2$, it is by \cite[4.3.1, Theorem 4.1.6]{maulik2022picard}. Then the estimate on $Z(m).C$ follows by noting the difference between $-(\omega.C) + \sum_{m=1}^{\infty} Z(m).C$ and the $\mathfrak{e_0}$ component of $-(\omega.C)$ is a cusp form, and applying Deligne's Weil bound on cusp form coefficients (see \cite{deligne1973formes,deligne1974conjecture}) for $b = 2$ and the trivial bound for $b \geq 3$ (see \cite{sarnak1990some}). The final claim follows from \cite[Theorem 7.4]{maulik2022reductions} for $b \geq 3$ and by a divisor sum argument similar to Lemma 4.3.4 of \cite{maulik2022picard}. 
\end{proof}

Thus we have 
\begin{equation}\label{eq.globalbound}
    i_P(C.Z(m)) \leq (C.Z(m)) \ll_{C,\epsilon} \begin{cases}
    m^{1 + \epsilon} & b = 2\\
    m^{b/2} & b \geq 3
    \end{cases}
\end{equation}

Finally, for a point $P \in (C \cap Z(m))(\bar{\mathbb{F}}_q)$, we introduce the notion of the \emph{global intersection number at $P$} in order to compare the global intersection with the local intersection numbers. 

\begin{defn}[{\cite[Def 7.1.3]{maulik2022picard}}]\label{def.globalint}
    Let $H$ denote the Hasse-invariant on $\mathscr{S}_{\bar{\mathbb{F}_p}}$, which cuts out the non-ordinary locus. On the formal neighborhood $\hat{C}_P = \operatorname{Spf} \bar{\mathbb{F}}_p[[t]]$, let $h_P = v_t(H)$ the t-adic valuation of $H$ on $\hat{C}_P$. Define 
    \[g_P(m) = \frac{h_P}{p-1} |q_L(m)|. \]
    In particular, $g_P(m) = 0$ for $P$ ordinary, and 
    \[\sum_{P \in (C \cap Z(m))(\bar{\mathbb{F}}_p)} g_P(m) = |q_L(m)| (\omega.C) \]
    since $H$ is a section of $\omega^{p-1}$. 
\end{defn}

\subsection{Local Intersection}

We now prove that at a fixed point as we take a larger family of $m$, the magnitude of the local intersection number grows at a much slower rate than the global intersection. To count lattice points of bounded norm across a growing family of lattices, we utilize a geometry of numbers method. First, we recall a definition as in \cite{eskin1995singular}. 

\begin{defn}
    For $1 \leq i \leq r = \mathrm{rk}_{\mathbb{Z}} L_n$ we define the \emph{$i$\textsuperscript{th} successive minimum} $\mu_i(n)$ of $L_n$ to be 
    \[\inf \{y \in \mathbb{R}_{>0} : \exists v_{1}, \dots, v_i \in L_n \text{ linearly independent, and } Q(v_j) \leq y^2, 1 \leq j \leq i  \} \]
    Then define $a_i(n) = \prod_{j=1}^i \mu_j(n)$. 
\end{defn}

We have the following lower bounds on $a_i(n)$ that is uniform in $n$ and the rank of our lattice of special endomorphisms. 

\begin{lemma}\label{lemma.sucmin}
    For $b \geq 2$, we may bound $a_i(n)$ by 
    \[a_i(n) \gg_{C, \epsilon_1}  \begin{cases}
    n^{\frac{i}{2 + 2\epsilon}} & b = 2\\
    n^{i/b} & b \geq 3
    \end{cases}\] 
    where we may take $\epsilon_1 > 0$ for $b = 2$ and $\epsilon_1 = 0$ for $b \geq 3$.
\end{lemma}

\begin{rem}
    We note that for $b = 2$, the bound in \ref{eq.globalbound} introduces an epsilon loss in our bound on the successive minima. We will see that this is harmless in the non-supersingular case. However, we will have to be more careful at supersingular points, where the rank of the special endomorphism lattices no longer automatically imply that the local intersection contribution is of smaller order than the global intersection. 
\end{rem}

\begin{proof}
    Let $a_1(n) = m^2$ and $v \in L_n$ be a nonzero vector attaining this minimum. Then by definition note that $v \in L_j$ for $j \leq n$, whence by definition of the local intersection number, $i_P(C.Z(m)) \geq n$. But from equation \ref{eq.globalbound}, there exist constants $c, c'$ such that $n < c m^{1 + \epsilon}$ for $b = 2$ and $n < c' m^{b/2}$ for $b = 3$. By definition of $m$, we reach our desired conclusion for $i = 1$. The bounds for $a_i(n)$ follow from the simple observation $\mu_i(n) \geq a_1(n)$ and thus $a_i(n) \geq a_1(n)^{i}$.
\end{proof}

This result on the successive minima allows us to bound the maximum level $n$ such that the lattice $L_n$ represents a particular value of $m$. 

\begin{corollary}\label{cor.upper}
    For $S_{D,M}$ defined above, there exists a constant $c$ depending only on $C$ such that 
    \[\sum_{m \in S_{D,M}} i_P(C.Z(m)) = \sum_{n=1}^{c \, U(X)} \sum_{m \in S_{D,M}} \# \{v \in L_n : Q(v) = m \} \]
    where the upper limit function $U(M) = M^{1 + \epsilon}$ for $b = 2$ and $U(X) = X^{b/2}$ for $b \geq 3$.
\end{corollary}

\begin{proof}
    First, suppose $b \geq 3$. Then by Lemma \ref{lemma.sucmin} we note that there exists $c$ only depending on $C$ such that for $n > c X^{b/2}$, we have $a_1(n) > (2X)^{1/2}$. Thus, the minimum value of $Q(v)$ for $v \in L_n$ is at least $2X$, and we conclude. For $b = 2$, the same proof works taking $n > c X^{1 + \epsilon}$ as above. 
\end{proof}

The data of successive minima their products allow us to bound our representation numbers across our varying families of $L$ in a way that only depends on the rank of $L$. We leverage the fact that at a non-supersingular point, the rank is at most $b$ to deduce a sufficient upper bound on the local intersection number at said points. 

\begin{proposition}\label{prop.nonss}
    For a non-supersingular point $P$, we have 
    \[\sum_{m= 1}^{2M} i_P(C.Z(m)) = \begin{cases}
        O(M^{1+\epsilon}) & b = 2\\
        O(M^{b/2} \log X) & b \geq 3
    \end{cases} \]
    In particular, we have the same bounds on $\sum_{m \in S_{D,M}} i_P(C.Z(m))$. 
\end{proposition}

\begin{proof}
    Since $P$ is not supersingular, $\mathrm{rk}_{\mathbb{Z}} L_n \leq b$. By Corollary \ref{cor.upper} and a basic geometry of numbers argument (see \cite[Lemma 2.4, Equations (5),(6)]{eskin1995singular},)
    \[\sum_{m= 1}^{2M} i_P(C.Z(m))  = \sum_{n=1}^{c\,U(M)} \sum_{m=1}^{2M} \# \{ v \in L_n : Q(v) = m\} \ll \sum_{n=1}^{c\,U(M)} \sum_{i=0}^r \frac{2M^{i/2}}{a_i(n)} \]
    where the implicit constant is absolute. We now apply our bound from Lemma \ref{lemma.sucmin}. 

    For $b = 2$, we have 
    \[ \sum_{n=1}^{c M^{1+\epsilon}} \sum_{i=0}^r \frac{2M^{i/2}}{a_i(n)} \ll_C \sum_{n=1}^{c M^{1+ \epsilon}} \sum_{i=0}^r \frac{(2M)^{i/2}}{n^{\frac{i}{2+2\epsilon}}}\] 
    Using that $r \leq 2$, 
    \[\sum_{m=1}^{2X} i_P(C.Z(m)) \ll M\sum_{n=1}^{c M^{1+\epsilon}} n^{-\frac{1}{1 + \epsilon}} + M^{1/2} \sum_{n=1}^{c M^{1 + \epsilon}} n^{-\frac{1}{2+2\epsilon}} + M^{1 + \epsilon}. \]
    Using the standard estimate on sums of powers, we conclude this is $O(M^{1 + \epsilon})$.

    For $b \geq 3$, for $1 \leq i < b$ the same argument yields the bound
    \[\sum_{n=1}^{c M^{b/2}} \frac{{2M}^{i/2}}{n^{i/b}} = (2M)^{i/2} \sum_{n=1}^{c M^{b/2}} \frac{1}{n^{i/b}} = O(M^{b/2}) \]
    On the other hand, for $i = b$, using the bound on the harmonic sum 
    \[\sum_{n=1}^{cM^{b/2}} \frac{1}{n} = O(\log(M^{b/2})) \]
    we may bound the sum by $O(M^{b/2} \log M)$.
    
\end{proof}

\subsubsection{Supersingular Points}
We now bound the local intersection at a fixed supersingular point $P \in C$. For a fixed $T \in \mathbb{Z}_{> 0}$, we break our intersection number into two a main and error term
\[\sum_{m \in S_{D,M}} i_P(C.Z(m)) = i_P(M,T)_{main} + i_P(M,T)_{err} \]
where we define 
\begin{align*}
    i_P(M,T)_{main} = \sum_{n=1}^T \sum_{m \in S_{D,M}} \#\{v \in L_n | Q(v) = m\}\\
    \label{eq.error} i_P(M,T)_{err} = \sum_{n = T}^{cU(X)} \sum_{m \in S_{D,M}}  \#\{v \in L_n | Q(v) = m\}. 
\end{align*}

To bound the error term, we state a bound on the representation numbers $r_n(m)$ using a more careful geometry of numbers argument. Notably, this result holds for a general positive definite lattice of rank $r \geq 2$.

\begin{proposition}\label{prop.square}
    Suppose the lattice of special endomorphisms $L_n$ at a point $P$ has rank $r$. Then for a fixed $D \in \mathbb{Z}_{>0}$, we have the bound 
    \[\sum_{m = D \ell^2, \ell \text{ prime }, m \leq M } r_n(m) = O_{\delta}\left(M^{\delta} \sum_{i = 0}^{r-1} \frac{M^{i/2}}{a_i(n)}\right)\]
\end{proposition}

\begin{proof}
    The same argument as that of \cite[Proposition 9.1.3(a)]{maulik2022picard} holds, their argument being the case where $r = 5$. 
\end{proof}

\begin{proposition}\label{prop.errorterm}
    We have the following bound on the error term 
    \[i_P(M,T)_{err} = O_{\epsilon,D}\left(\frac{M^{(b+1)/2 + \epsilon}}{T^{1/b}} + M^{b/2 + \epsilon}\log M \right)\]
    In particular, choosing $T \asymp M^{\epsilon_0}$ for sufficiently small $\epsilon_0$, we have $i_P(M,T)_{err} = o(M^{(b+1)/2})$.  
\end{proposition}

\begin{proof}
    Fix $\epsilon > 0$. We prove the $b = 2$ case first. First note that the rank of the special endomorphism lattice at $P$ must be 4 as in \cite[8.1]{maulik2022picard}. Thus, the rank of the special endomorphism lattice must be $4$. Thus Lemma \ref{lemma.sucmin} and Proposition \ref{prop.square} give
    \[i_P(M,T)_{err} \ll_{\delta, \eta, D} M^{\delta} \sum_{n=T}^{c_{\eta} M^{1+\eta}} \left(\frac{M^{3/2}}{n^{\frac{3}{2 + 2\eta}}} + \frac{M}{n^{\frac{1}{1 + \eta}}}+ \frac{M^{1/2}}{n^{\frac{1}{2 + 2\eta}}} + 1 \right) \]
    Note that the latter terms are all $O(M^{1 + \delta + \eta})$. For the first term, note that $0 < \eta < \frac{1}{2}$ implies $\frac{3}{2 + 2\eta} - 1 > 0$ and thus 
    \[M^{3/2 + \delta} \sum_{n=T}^{c_{\eta}M^{1 + \eta}} \frac{1}{n^{\frac{3}{2 + 2 \eta}}} \ll_{\eta} \frac{M^{3/2 + \delta}}{T^{\frac{3}{2 + 2\eta} - 1}}\]
    To reach the form in the statement of the Proposition, note that since $T \ll M^{1 + \eta}$, we have
    \[T^{1 - \frac{3}{2 + 2\eta}} = T^{-1/2} T^{3/2 - \frac{3}{2 + 2\eta}} \ll T^{-1/2} M^{3\eta/2}\]
    Choose $\delta, \eta$ so that $\delta + \frac{3\eta}{2} < \epsilon$, and we conclude that 
    \[i_P(M,T)_{err} \ll \frac{M^{3/2 + \epsilon}}{T^{1/2}} + M^{1 + \epsilon} \]
    as desired. 

    For $b \geq 3$, the same proof carries through. Noting that the rank of our special endomorphism lattice is $b+2$, we have 
    \[i_P(M,T)_{err} \ll M^{\delta} \sum_{i=0}^{b-1} M^{i/2} \sum_{n=T}^{c M^{b/2}} \frac{1}{n^{i/b}}  \]
    As in the proof of Proposition \ref{prop.nonss}, we note that for $i < b$, the terms have an $O(M^{b/2})$ contribution. For $i = b$, the there is a $O(M^{b/2} \log(M))$ contribution, and for $i = b+1$, we have a $O(M^{(b+1)/2}/T^{1/b})$ contribution. Putting these terms together, we conclude. 
\end{proof}

For the main term, let $\Theta_n$ be the theta series attached to the lattice $L_n$. By definition, its $m$\textsuperscript{th} Fourier coefficient $r_n(m) \geq \# \{v \in L_n : Q(v) = m \}$ and thus 
\[i_P(M,T)_{main} \leq \sum_{n=1}^T \sum_{m \in S_{D,M}} r_n(m). \]
Decompose $\Theta_n = E_n + G_n$ into a sum of Eisenstein and cusp forms, and let $q_{E_n}(m)$ and $q_{G_n}(m)$ denote their $m$th Fourier coefficients, respectively. We thus have 

\[    i_P(M,T)_{main} \leq \sum_{n=1}^T \sum_{m \in S_{D,M}} (q_{E_n}(m) + q_{G_n}(m))\]

We first focus on bounding the $q_{G_n}(m)$ terms. Since the number of cusp forms we are summing depends on our choice of $T$, we state the following uniform bound.

\begin{lemma}[Waibel]\label{lemma.waibel}
	Let $Z$ be a finite set of primes, and let $\Theta$ be the theta series attached a positive definite quadratic lattice $L_{\Theta}$ of rank $b+2 \geq 4$. Let the discriminant of $L_{\Theta}$ be $D_{\theta}$.  to $\theta = E + G$. Let $N$ be the level of $G$, and suppose that $(m,N) = O(1)$. Then for all $\Theta$ such that all prime factors of $D_{\theta}$ are in $Z$, we have the $m$\textsuperscript{th} Fourier coefficient of $G$ satisfies 
	\[q_G(m) \ll_b D_{\theta}^{\frac{b+1}{4}} m^{\frac{b+1}{4} + \frac{1}{12}} \]
    for $m \in Z$. 
\end{lemma}

\begin{proof}
	From equation (3) in \cite{waibel2022uniform} we have the following uniform bound on the coefficients of a weight $\frac{b+2}{2}$ cusp form $G$ 
	\[q_G(m) \ll_{b, \epsilon} ||G|| m^{\frac{b}{4}} \left(1 + \frac{m^{1/4}(m,N)^{1/4}}{N^{1/2}}\right)(mN)^{\epsilon}, \] 
	where $N$ is the level of $G$ and $|| \cdot ||$ denotes the Petersson norm. 
    
    Since we assumed $(m,N) = O(1)$, we have  
	$q_G(m) \ll_{\epsilon} ||G|| m^{\frac{b+1}{4} + \epsilon} N^{\epsilon}$
	We note that Theorem 1 of \emph{loc. cit.} gives us the necessary bounds on $||G||$ and conclude.
\end{proof}

This bound together with some control on the lattice discriminants, we arrive at the following bound on the term of cusp forms. Notably, our choice of $S_{D,M}$ satisfy the lemma's assumptions. 

\begin{proposition}\label{prop.cusp}
    The sum of cusp coefficients is bounded by 
    \[\sum_{n=1}^T \sum_{m \in S_{D,M}} q_{G_n}(m) \ll_{\eta,D} \begin{cases}
        M^{4/3} T^{11/2} & b = 2 \\
        M^{\frac{b+3}{4} + \frac{1}{12}} T^{b+2} &  b \geq 3
    \end{cases}\]
    In particular, choosing $T \asymp M^{\epsilon_0}$ for sufficiently small $\epsilon_0$ implies that $\sum_{n=1}^T \sum_{m \in S_{D,M}} q_{G_n}(m) \ll o(M^{(b+1)/2})$
\end{proposition}

\begin{proof}
    Suppose that $b \geq 3$. For $r \geq 0$, Theorems 3.2 and 3.3 of \cite{maulik2022reductions} imply that for $n \asymp p^r$, we have an enlargement $\tilde{L}_n'$ of $L_n'$ such that 
    \[[L_1': \tilde{L}_n'] \leq p^{3 + 2r} \ll_P (p^r)^2 \ll_P n^2 \]
    and thus $\mathrm{disc} \tilde{L}_n' \ll n^4$. Thus by Lemma \ref{lemma.waibel}, 
    \[q_{G_n}(m) \ll_{b,P} n^{b+1} m^{\frac{b+1}{4} + \frac{1}{12}} \]
    Summing across $n \leq T$, we have  
    \[\sum_{n=1}^T q_{G_n}(m) \ll m^{\frac{b+1}{4} + \frac{1}{12}} \sum_{n=1}^T n^{b+1} \ll m^{\frac{b+1}{4} + \frac{1}{12}} T^{b+2}.  \]
    Finally, using the fact that $\# S_{D,M} \ll M^{1/2}$, we have 
    \[\sum_{n=1}^T \sum_{m \in S_{D,M}} q_{G_n}(m) \ll M^{\frac{b+3}{4} + \frac{1}{12}}T^{b+2}   \]
    For $b = 2$, a similar argument holds as in \cite[\S 9.2, Small $n$'s]{maulik2022picard} using the bound $\mathrm{disc} \tilde{L}_{n,i}' \ll p^{6k}$, as in their notation. See Theorem 5.1.2, \S 7.2.3 of \emph{loc. cit.}
\end{proof}

\begin{proposition}\label{prop.mainterm}
    There exists an absolute constant $0 < \alpha < 1$ such that  
    \[ i_P(M,T)_{main} \leq \alpha \sum_{m \in S_{D,M}} g_P(m) + o(M^{(b+1)/2}) \]    
\end{proposition}

\begin{proof}
    By Proposition \ref{prop.cusp}, it suffices to show 
    \[\sum_{n=1}^T q_{G_n}(m) \leq \alpha g_P(m) \]
    for some constant $0 < \alpha < 1$. This follows from the proof of Proposition 7.17 in \cite{maulik2022reductions} for $b \geq 3$ and it is the computation in \cite[\S9.2, Small $n$'s]{maulik2022picard} for $b = 2$. In particular, using the Siegel mass formula they prove these estimates are uniform in $T$ for individual admissible $m$. 
\end{proof}

Putting Propositions \ref{prop.errorterm} and \ref{prop.mainterm} together, we have the desired bound on the local intersection number at supersingular points. 
\begin{proposition}\label{prop.sstotal}
    There exists an absolute constant $0 < \alpha' < 1$ such that for $S_M$ as above and any $P \in C(\bar{\mathbb{F}_p})$ supersingular, we have 
    \[\sum_{m \in S_M} i_P(C.Z(m)) \leq \alpha' \sum_{m \in S_M} g_P(m) + o(M^{(b+1)/2}) \]
\end{proposition}

\begin{proof}
    Take $1 > \alpha' > \alpha$, with $\alpha$ as in Proposition \ref{prop.mainterm}. Choose $T = \lfloor M^{\epsilon_0} \rfloor$ with $\epsilon_0$ small enough so error term contribution in Proposition \ref{prop.errorterm} and the cusp form contribution in Proposition \ref{prop.cusp} are both $o(M^{(b+1)/2})$. Then the result follows from Proposition \ref{prop.mainterm}.
\end{proof}

\begin{proof}[Proof of Theorem \ref{thm.infinitude}]
    Suppose for contradiction that there are only finitely many points $P$ in $C \cap \bigcup Z(Dm^2)$. Then by Propositions \ref{prop.nonss}, \ref{prop.mainterm}, \ref{prop.errorterm}, \ref{prop.sstotal}, and Definition \ref{def.globalint} we have 
    \[\sum_{m \in S_M} C.Z(m) = \sum_{m \in S_M} \sum_{P \in C \cap (\bigcup_{m \in S_M} Z(m))(\bar{\mathbb{F}}_p)} i_P(C.Z(m)) = \alpha' (\omega.C) \sum_{m \in S_M} |q_L(m)| + o(M^{(b+1)/2}) \]
    but this contradicts Proposition \ref{prop.global}. 
\end{proof}

\section{Density of Special Cycles with a Prescribed Norm}

In this section, we show that having a new line bundle $L$ with $Q(L) = Dm^2$ for some fixed $D$ is not always a density 0 condition, see \S \ref{sec.four}. On the other hand, we prove that with specific restrictions it is a density 0 condition. 

\subsection{Our strategy}\label{sec.four}

In our proof of Theorem \ref{thm.main}, the Tate conjecture allowed us to study the eigenvalues of possible Frobenii in the orthogonal group to determine if there were extra algebraic cycles. Unfortunately, this method is agnostic to whether these Galois-invariant cycles satisfy a specific norm condition. 

To make our finer estimates, we use a sieve method similar to \cite{chavdarov1995generic} and \cite{zywina2014splitting}. Roughly speaking, for each auxiliary prime $\ell$ we construct a conjugacy-invariant subset $C_{D,\ell}$ of the residual monodromy group with the property that having an extra line bundle of the prescribed norm at a point $p$ requires that the image of $Frob_p$ lands in $C_{D,\ell}$. The objective is to then show that $C_{D,\ell}$ occupies at most a proportion $1-\delta_{D,\ell}$ of each relevant Frobenius coset, for some $\delta_\ell>0$. Independence of the residual representations at distinct auxiliary primes bounds the proportion of Frobenius elements satisfying all these necessary conditions by
\[
\prod_{\ell\in S}(1-\delta_{D,\ell})
\]
for every finite set $S$ of auxiliary primes. By choosing $S$ so that this product tends to zero as $S$ grows, we obtain the desired density estimate. 

At a place of good reduction $p$ consider the space 
\[W_p = \mathrm{sp}(\NS(\bar{X}))^{\perp} \cap \NS(\bar{X}_p)   \]
of new cycles of $X_p$ at $p$. Then the bilinear form $-\langle \, , \rangle$ restricted to $W_p$ is positive definite by the Hodge index theorem.

We first apply a straightforward observation on quadratic forms to this question:

\begin{lemma}\label{lemma.meyer}
    Suppose $p$ is a place of $K$ at which $X$ acquires at least 4 new linearly independent cycles, that is, $\dim W_p \geq 4$. Then for each choice of $D \in \mathbb{Z}$, there exists some $m$ such that there is a line bundle $L_D \in W_p$ such that $Q(L_D) = Dm^2$. 
\end{lemma}

\begin{proof}
    Consider the $\mathbb{Q}$ vector space $W_{p,\mathbb{Q}}$ equipped with the negative intersection pairing, which we also denote $Q$. For each $D$, the existence of an integral element $L_D$ satisfying $Q(L_D) = Dm^2$ is equivalent to some element of $W_{p, \mathbb{Q}}$ of the form $L_D/m$ representing $D$. Thus it suffices to prove that each $D$ is represented by $q$ for some element in $W_{p, \mathbb{Q}}$. 
    
    To this end, note that $Q$ represents $D$ if and only if the form 
    \[Q \perp \langle -D \rangle \]
    has an isotropic vector. The latter is an indefinite form of rank at least 5, and we conclude by invoking Meyer's Theorem (see \cite[Chapter IV, \S 3, Corollary 2]{serre2012course})
\end{proof}

Following directly from Theorem \ref{thm.main} and Lemma \ref{lemma.meyer}, we have the 

\begin{corollary}
    If there exists a component $b \in \pi_0(G)$ such that $J(b) \geq 4$, then there exist a positive density of places $p$ of $K$ such that $X_p$ intersects the family of special divisors $\bigcup_{m} Z(Dm^2)$ for all squarefree integers $D$. 
\end{corollary}

In particular, Examples \ref{ex.jacobian} and \ref{ex.product} satisfy the assumptions of the lemma. Below, we discuss discuss a case in which we get a density 0 result. This case notably holds for a generic K3 surface, i.e. for most moduli points. 

\subsection{Equidistribution and open image theorems}

We first study the image of the product of mod $\ell$ representations and show that their image is large in some sense. We will then introduce the relevant equidistribution results on these images that will be vital to our estimates. We state the results below in full generality, regardless of the generic Picard rank jump on a given component.

\subsubsection{Number Field Case}
In the number field case, the following result states that after replacing $K$ by a finite field extension, the $\ell$-adic monodromy group is equal to the $\mathbb{Z}_{\ell}$-points of the Mumford-Tate group for almost all $\ell$. 

\begin{theorem}[Number Field Open Image, {\cite[Theorem 6.6]{cadoret2020integral}}]\label{thm.NFopen}
    Let $X$ be a K3 surface defined over a number field $K$, let $H = H^2_B(X(\mathbb{C}), \mathbb{Z})(1)$ and $\MT(H) \subset \GL(H)$ denote the Mumford-Tate group. Then the image of 
    \[\rho_X : \Gal(\bar{K}/K) \rightarrow \GL(H)(\hat{\mathbb{Z}}) \]
    has a subgroup of finite index which is an open subgroup of $\MT(H)(\hat{\mathbb{Z}})$.
\end{theorem}

Thus, there exist a finite Galois extension $K'/K$ and a finite set $\Sigma$ of rational primes such that for every finite set $S \subset \Sigma^c$, we have 
\begin{equation}\label{eq.ladicimagesNF}
    \mathrm{im}\left(G_{K'} \xrightarrow{(\rho_{\ell})_{\ell \in S}} \prod_{\ell \in S} \MT(H)(\mathbb{Z}_{\ell}) \right) = \prod_{\ell \in S} \MT(H)(\mathbb{Z}_{\ell}).
\end{equation}
For ease of notation, for a set of rational primes $S$ let $\lambda_S$ denote the map above. After possibly enlarging $\Sigma$, a similar fact holds for the reduction of $\lambda_S$ and $\mathbb{F}_{\ell}$ points
\begin{equation}\label{eq.imagesNF}
    \mathrm{im}(\overline{\lambda_S}(G_{K'})) = \prod_{\ell \in S} \MT(H)(\mathbb{F}_{\ell}) = \prod_{\ell \in S} \mathrm{im}(\bar{\rho}_{\ell}) 
\end{equation}

This leads to the following straightforward application of the Chebotarev density theorem. 

\begin{proposition}[{\cite[Proposition 2.12]{zywina2014splitting}}]\label{prop.Zywina}
    Fix a finite extension $K'/K$ so that 
    \[\rho_X(\Gal(\bar{K}/K')) \subseteq \MT(H)(\hat{\mathbb{Z}})\] 
    satisfying \ref{eq.ladicimagesNF} and fix $\Sigma$ as above. Let $S$ be a finite set of rational primes in $S \subset \Sigma^c$. For each $\ell \in S$, fix a subset $C_{\ell}$ of the group $\bar{\rho}_{\ell}$ that is stable under conjugation. Let $T$ be the set of primes $v$ of $K$ such that $\bar{\rho}_{\ell}(Frob_v) \subseteq C_{\ell}$ for all $\ell \in S$. Then $T$ has density 
    \[\delta(T) = \sum_{b} \frac{|b|}{|\Gal(K'/K)|} \cdot \prod_{\ell \in S} \frac{|\bar{\rho}_{\ell}(\Gamma_b) \cap C_{\ell}|}{|\bar{\rho}_{\ell}(\Gamma_b)|} \]
    where $b$ varies over the conjugacy classes of $\Gal(K'/K)$ and $\Gamma_b$ is the set of $\sigma \in \Gal_K$ for which $\sigma|_{K'} \in b$. 
\end{proposition}

\begin{proof}
    The same proof as in \emph{\frenchspacing{loc. cit.}} holds. In particular, the definition of $\Sigma$ and \ref{eq.imagesNF} fill the same role as the input of independence (\cite[138]{serre2013oeuvres}) in Zywina's proof. 
\end{proof}

\subsubsection{Function Field Case}
When $K$ is a global function field, we use the following general result to control the mod $\ell$ image of the geometric fundamental group under $\bar{\rho}_{\ell}$. 

\begin{theorem}[Function Field Open Image, {\cite[Theorem 7.3, (7.3.2)]{cadoret2017geometric}}]\label{thm.FFopen}
Let $Y$ be a connected, smooth, separated scheme over an algebraically closed field $k$, and let $f\colon Z\to Y$ be a smooth proper morphism of schemes. There exists a connected finite étale cover $Y'\to Y$ such that the following holds. Fix a geometric point $\bar{y}\to Y'$ and write $\Pi_{\ell,\infty}$ for the image of
\[
\rho_\ell\colon \pi_1(Y',\bar{y})
\rightarrow
\operatorname{GL}\left((H^*(Z_{\bar{y}},\mathbb{Z}_\ell)\right),
\]
and $\mathcal{G}_{\ell}$ for its Zariski closure. There exists a constant $C_r\geq 1$, depending only on $r=\mathrm{rank}_{\mathbb{Z}_\ell} H^*(Z_{\bar{y}},\mathbb{Z}_\ell)$
such that $[\mathcal{G}_{\ell}(\mathbb{Z}_\ell):\Pi_{\ell,\infty}] \leq C_r$
for all $\ell\gg 0$. In particular, for $\ell \gg 0$, $\Pi_{\ell,\infty} =\mathcal{G}_{\ell}(\mathbb{Z}_\ell)^+$, where $\mathcal{G}_{\ell}(\mathbb{Z}_\ell)^+$ denotes the closed subgroup generated by the Sylow $\ell$-subgroups.
\end{theorem}

Unlike Theorem \ref{thm.NFopen}, this result is for an individual $\ell$ rather than products of $\ell$. Nevertheless, we may use group theory to conclude a result for various finite products that is sufficient for our purposes. Let $\bar{\lambda}_S$ denote the product of $(\bar{\rho}_{\ell})$

\begin{proposition}\label{prop.geometric}
    Let $\mathcal{X} \rightarrow U$ be a generically ordinary one dimensional non-isotrivial family of K3 surfaces over a finite field $\mathbb{F}_q$. There exists a finite étale cover $U' \rightarrow U_{\bar{\mathbb{F}}_q}$ as in Theorem \ref{thm.FFopen} such that for any finite set $S$ of rational primes $\ell \gg 0$ we have 
    \[ \mathrm{im}(\overline{\lambda_S}|_{\pi_1(U', \bar{\eta})}) = D_S := \begin{cases}
        \prod_{\ell \in S} (\Res_{O_F \otimes \mathbb{Z}_{\ell}/\mathbb{Z}_{\ell}} SO(\Lambda_{F,\ell}, \psi)(\mathbb{F}_{\ell}))^{der} & \text{ if } F \text{ is totally real.} \\
        \prod_{\ell \in S} (\Res_{O_F^+ \otimes \mathbb{Z}_{\ell}/\mathbb{Z}_{\ell}} U(\Lambda_{F,\ell}, \psi)(\mathbb{F}_{\ell}))^{der} & \text{ if } F \text{ is CM.}
    \end{cases} \]
    where $\bar{\eta} \rightarrow U'$ is a geometric generic point, and $\lambda_S = (\rho_{\ell})_{\ell \in S}$. In particular, $D_S \subseteq \mathrm{im}(\lambda_S|_{\pi_1(U, \bar{\eta})^{geom}})$ and the index $[\mathrm{im}(\lambda_S|_{\pi_1(U, \bar{\eta})^{geom}}): D_S]$ is bounded uniformly for a fixed family $\mathcal{X} \rightarrow U$. 
\end{proposition}

\begin{proof}
    We recall that for a connected reductive group $G$ defined over $\mathbb{F}_{\ell}$ for $\ell \geq 5$, we have $G(\mathbb{F}_{\ell})^+ = G(\mathbb{F}_{\ell})^{der}$ by \cite[Corollaire 6.4]{borel1973homomorphismes}. Then the first assertion follows from Theorem \ref{thm.FFMT}, Zarhin's classification in Proposition \ref{prop.centralizer} and Theorem \ref{thm.FFopen}, by taking $U' \rightarrow U$ to be large enough so that the monodromy group is connected and then applying Goursat's Lemma \cite[5.2]{ribet1976galois}. The second claim is by the choice of our finite étale cover as in Theorem \ref{thm.FFopen}.
\end{proof}

The relevant version of Chebotarev density is the following

\begin{proposition}[{\cite[Theorem 4.1]{chavdarov1995generic}}]\label{prop.chavdarov}
    Let $U/\mathbb{F}_{q}$ be a geometrically irreducible variety and assume that we have a commutative diagram
\[\begin{tikzcd}
	1 & {\pi_1^{geom}} & {\pi_1} & {\hat{\mathbb{Z}}} & 1 \\
	1 & {G_0} & G & \Gamma & 1
	\arrow[from=1-1, to=1-2]
	\arrow[from=1-2, to=1-3]
	\arrow["{\lambda_0}", two heads, from=1-2, to=2-2]
	\arrow[from=1-3, to=1-4]
	\arrow["\lambda", from=1-3, to=2-3]
	\arrow[from=1-4, to=1-5]
	\arrow["{1 \mapsto -\gamma}", from=1-4, to=2-4]
	\arrow[from=2-1, to=2-2]
	\arrow[from=2-2, to=2-3]
	\arrow["mult"', from=2-3, to=2-4]
	\arrow[from=2-4, to=2-5]
\end{tikzcd}\]
where $G$ is a finite group, $\Gamma$ is abelian, and $\lambda_0$ is surjective. Let $C \subset G$ be stable under $G$-conjugation and let $G^{(\gamma^n)} = \{g \in G | mult(g) = \gamma^n \}$. Then 
\[\left|\mathrm{Prob} \{u \in U(\mathbb{F}_{q^n}) |\lambda(\mathrm{Frob}_{\mathbb{F}_{q^n},u}) \in C \} - \frac{\#(C \cap G^{(\gamma^n)})}{\#G_0}  \right| \leq O\left(\frac{1}{\sqrt{q^n}}\right) \]
where the implicit constant only depends on $U$. 
\end{proposition}

\subsection{Generic Jump 1}
We consider the case when the subset 
\[\{p | \rho(\overline{X}_p) = \rho(\overline{X}) + 1 \} \]
of places of $K$ is of density 1. By Theorem \ref{thm.main}, the only possible case here is that $F = \mathbb{Q}$ and that $\dim V$ is odd. This condition corresponds to the Mumford-Tate group being as large as possible, and roughly speaking is true for ``most'' moduli points. 

Note that in this case, the neutral component of the algebraic monodromy group is the full special orthogonal group $SO(V_{\ell}, \psi_{\ell})$. This means that generically an element in the determinant $\epsilon = \pm 1$ component has exactly one eigenvalue equal to a root of unity, namely $\epsilon$. Let $U_{\ell}$ denote the set of elements satisfying this simple eigenvalue condition. In the notation of \ref{eq.mult}, 
\[U_{\ell} = \{x \in O(V_{\ell}, \psi_{\ell}): m(x) = 1\}.  \]
Then for a prime $p$ corresponding to $x = \mathrm{Frob}_p \in U_{\ell}$, we have the extra cycle space is a line $L_x$. The quadratic form $Q$ restricted to $L_x$ determines a single squareclass $a_{x,\ell} \in \mathbb{Q}_{\ell}^{\times}/(\mathbb{Q}_{\ell}^{\times}))^2$ coming from a single rational square-class $a_x \in \mathbb{Q}^{\times}/(\mathbb{Q}^{\times})^2$ since the conjugacy class of Frobenius is defined over $\mathbb{Q}$. 

Let $p_x$ be the characteristic polynomial of an element $x$ and  
\[r(x) = \mathrm{ord}_{T-1} (p_x) + \mathrm{ord}_{T+1} (p_x)\]
denote the total algebraic multiplicity of the eigenvalues $1$ and $-1$. Fix an integral quadratic lattice $(\Lambda,Q)$ in $V_{\ell}$. After excluding finitely many primes $\ell$, the reduction of $Q$ on $\bar{V}_{\ell} = \Lambda/\ell \Lambda$ is nondegenerate. For such $\ell$, let
\[\bar{U}_{\ell} = \{x \in U_{\ell} : r(x) = 1 \}. \]
For such $x$, its reduced eigenline is nondegenerate. Consequently for such $x$ and $\ell$, $a_{x,\ell}$ is represented by an element of $\mathbb{Z}_{\ell}^{\times}$. 

Now fix a positive integer $D$. The condition that at a point $p$
\[\iota^{KS}(p) \in \bigcup_{m \geq 1} Z(Dm^2) \]
means that $x = \mathrm{Frob}_p$ has a fixed vector $v$ satisfying $[a_x] = [D] \in \mathbb{Q}^{\times}/(\mathbb{Q}^{\times})^2$. Having the same square class in $\mathbb{Q}^{\times}$ implies the same square class in $\mathbb{Q}_{\ell}^{\times}$. Thus after  we additionally impose the condition that $\ell \nmid D$, this implies $a_{x,\ell} = D \in \mathbb{Z}_{\ell}^{\times}/(\mathbb{Z}_{\ell}^{\times})^2$ and after reduction mod $\ell$, they must represent the same squareclass in $\mathbb{F}_{\ell}^{\times}$. For convenience, we introduce the notation 
\[\chi_{\ell,D}(a_{x,\ell}) = \begin{cases}
    1 & \text{ if } a_{x,\ell}, D \text{ represent the same square class mod } \ell\\
    -1 & \text{ otherwise }
\end{cases} \]
Then define 
\[C_{D,\ell} = (\mathcal{G}_{\ell}(\mathbb{F}_{\ell}) \setminus \bar{U}_{\ell}) \cup \{x \in \bar{U}_{\ell} | \chi_{\ell,D} = 1 \}. \]
\begin{lemma}
    $C_{D,\ell}$ is stable under conjugation and for any point $p \in Z(Dm^2)$, the image of its Frobenius element lands in $C_{D,\ell}$.  
\end{lemma}

\begin{proof}
    We claim $C_{D,\ell}$ is stable under conjugation. The first summand is, since characteristic polynomials are invariant under conjugation. For the latter, note $a_x = Q(u_x)$ for a generator $u_x$ of $L_x$. Then $L_{gxg^{-1}} = gL_x$ and thus is generated by $g u_x$. The statement follows by noting that $Q(gu_x)$ and $Q(u_x)$ represent the same square class. 

    The second part of the statement follows by noting that if the image of $Frob_p$ has $\pm 1$ with multiplicity at least two, it has image in the first summand of $C_{D,\ell}$. Otherwise, by the argument above, it must be an element of the second summand. 
\end{proof}

Let $\Omega_{\ell}$ denote the kernel of the map 
\[\Or(\bar{V}_{\ell}) \xrightarrow{\det \times \mathrm{sn}} \{\pm 1 \} \times \mathbb{Z}/2\mathbb{Z}  \]
where $\mathrm{sn}$ is the spinor norm. Note that for $\ell \geq 5$, we have the equality 
\[(\SO(\bar{V}_{\ell})(\mathbb{F}_{\ell}))^{der} = \Omega_{\ell} \]
by \cite[Chapter 11]{taylor1992geometry}. Thus, by the open image results above, the mod $\ell$-image always contains $\Omega_{\ell}$. We would like to show that for all choices of $\ell$ and cosets $b \Omega_{\ell}$, we have $x \in C_{D,\ell} \cap b\Omega$ for a proportion $\delta < 1$ of elements of $b \Omega_{\ell}$. 

The general idea is that each $x \in U_{\ell}$ corresponds uniquely to a choice of a fixed line $L \subset \bar{V}_{\ell}$ and to a choice of $B \in SO(L^{\perp})$ having no eigenvalue equal to $\pm 1$. We then note that the additional condition of membership in $b\Omega_{\ell}$ corresponds to a specific spinor norm condition that half of the elements of $B$ satisfy. We formalize this idea below. 

Let
\[\mathcal{L}_a = \{L \subset \bar{V}_{\ell}, \dim L = 1:  Q|_L \neq 0, [Q(L)] = a\} \]
denote the set of lines in $\bar{V}_{\ell}$ that represent a given squareclass $a$. 

This counting lemma shows that roughly $1/2$ of non-isotropic lines represent a given square class in $\mathbb{F}_{\ell}^{\times}$. 
\begin{lemma}\label{lemma.lines}
	Fix $L \in \mathcal{L}_a$. We have 
	\[\# \mathcal{L}_a = \frac{\#O(\bar{V}_{\ell})}{\#2O(L^{\perp})} \]
\end{lemma}

\begin{proof}
	The orthogonal group $O(V)$ acts transitively on $\mathcal{L}_a$. To see why, note that if $L = \langle u \rangle$ and $L' = \langle v \rangle$ have the same squareclass, rescale $v$ so that $Q(u) = Q(v)$. The isometry $u \mapsto v$ can be extended to an isometry of $\bar{V}_{\ell}$ via the Witt Isometry theorem. 

	On the other hand, the stabilizer of $L$ is $O(L) \times O(L^{\perp})$, which has order $2 (\#O(L^{\perp}))$. The statement then follows via Orbit-Stabilizer. 
\end{proof}

We then count the number of elements in each coset that correspond to a particular line. 

\begin{lemma}\label{lemma.B}
	The number of $x \in b\Omega \cap U_{\ell}$ such that $x$ corresponds to the line $L$ is 
	\[\frac{1}{2} \# SO(L^{\perp}) + O(\#SO(L^{\perp})/\ell),\]
	where the implicit constant only depends on $n = \dim V$ and the error term corresponds to the complement of $\bar{U}_{\ell}$. 
\end{lemma}

\begin{proof}
    Let the determinant of elements in $b \Omega_{\ell}$ be $\epsilon$. Following the discussion above, the number of such $x$ corresponds to the number of $B$ so that 
    \[x = (\epsilon I_L \perp B) \in b\Omega_{\ell} \cap U_{\ell}.\]
    For $x$ to be an element of $b \Omega_{\ell}$, $B$ must have the appropriate spinor norm. Exactly half of the elements of $\SO(L^{\perp})$ have a given spinor norm, which gives the first term in the expression. All that remains is to subtract the number of $B$ that have $\pm 1$ as an eigenvalue. Let 
    \[Z_{\pm 1} = \{B \in SO(L^{\perp}) | \det(B - I)  \text{ or } \det(B + I) = 0 \}. \]
    Since $L^{\perp}$ is even dimensional, $Z_{\pm 1}$ is a proper closed subvariety of $SO(L^{\perp})$. Thus by the Lang-Weil estimate, $Z_{\pm 1}(\mathbb{F}_{\ell}) = O(\ell^{\dim(SO(L^{\perp})-1}) = O(\#SO(L^{\perp})/\ell)$ (see \cite{lang1954number}).
\end{proof}

\begin{proposition}\label{prop.proportion}
The proportion of elements in each coset $b \Omega_{\ell}$ of $\Omega_{\ell}$ in our set $C_{D,\ell}$ is 
\[\frac{\# \{x \in b \Omega_{\ell} \cap C_{D,\ell}\}}{\# \Omega_{\ell}} = \frac{1}{2} + O(\ell^{-1}), \]
where the implicit constant only depends on $n$.
\end{proposition}

\begin{proof}
    By our preceding discussion, the size of $b \Omega_{\ell} \cap \{x \in U_{\ell} | \chi_{\ell,D} = 1 \}$ is the number of pairs $(L,B)$ with $L \subset \bar{V}_{\ell}$ a nonisotropic line representing a fixed squareclass $a \in \mathbb{F}_{\ell}^{\times}/(\mathbb{F}_{\ell}^{\times})^2$ and $B \in SO(L^{\perp})$ having no eigenvalues equal to $\pm 1$. Combining Lemmas \ref{lemma.lines} and \ref{lemma.B} we have 
    \begin{align*}
        \sum_{L \in \mathcal{L}_a} \#\{x \in b \Omega_{\ell} \cap U_{\ell} \text{ fixing } L \} &= \sum_{L \in \mathcal{L}_a} \frac{1}{2}\#SO(L^{\perp}) + O(\#SO(L^{\perp})/\ell) \\
        &= \frac{\# O(\bar{V}_{\ell})}{\# 2 O(L^{\perp})}\left(\frac{1}{2}\#SO(L^{\perp}) + O(\#SO(L^{\perp})/\ell) \right)\\
        &= \frac{\#O(\bar{V}_{\ell})}{8} + O(\# \Omega_{\ell}/\ell)\\
        &= \frac{\#\Omega}{2} + O(\# \Omega_{\ell}/\ell)
    \end{align*}
    To count $b\Omega_{\ell} \cap U_{\ell}^c$, note that we may use Lang-Weil as above to conclude that 
    \[\#(b\Omega \cap U_{\ell}^c) = O(\# \Omega_{\ell}/\ell) \]
    Dividing by $\# \Omega_{\ell}$, we reach our desired count.
\end{proof}

The fact that the proportion of elements in $C_{D,\ell}$ is uniform on each coset is exactly the type of input Propositions \ref{prop.Zywina} and \ref{prop.chavdarov} need, which average over various Galois cosets. 

\begin{proof}[Proof of Theorem \ref{thm.sqclass}]
    We start with the function field case. Let $X_{n,D}$ be the event that for $x \in U(\mathbb{F}_{q^n})$ we have the new cycle space at $x$ represents the squareclass $D$. 
    
    Let $S$ be a set of finite primes and let 
    \[G_{0,S} = \prod_{\ell \in S} G_{\ell}^{geom}, \quad G_S = \prod_{\ell \in S} \mathrm{im}(\bar{\rho}_{\ell})_{\ell \in S} \] 
    By Goursat's lemma, we have that the map $\Lambda_S := (\bar{\rho}_{\ell})_{\ell \in S}$ is surjective onto $G_{0,S}$.

    As in Propositon \ref{prop.geometric}, $G_{0,S}$ is a union of cosets of $D_S$ which we write as
    \[G_{0,S}  = \bigsqcup_{h \in G_{0,S}/D_S} h D_S. \]
    Fix $g_n \in G_S^{(\gamma^n)}$. Then we have 
    \[G_S^{(\gamma^n)} = g_n G_{0,S} =  g_n \bigsqcup_{h \in G_{0,S}/D_S} h D_S.\] 
    We may thus use Proposition \ref{prop.proportion} to estimate
    \begin{align*}
        \frac{\#(C_S \cap G_S^{(\gamma^n)})}{\#G_{0,S}} &= \frac{1}{|G_{0,S}/D_S|} \sum_{h \in G_{0,S}/D_S} \prod_{\ell \in S} \frac{\#(C_{\ell} \cap h\Omega_{\ell})}{\#\Omega_{\ell}}\\
        &\leq  \frac{1}{|G_{0,S}/D_S|} \sum_{h \in G_{0,S}/D_S} \prod_{\ell \in S} \left(\frac{1}{2} + \frac{A}{\ell}\right)\\
        & \leq \prod_{\ell \in S} \left(\frac{1}{2} + \frac{A}{\ell}\right)
    \end{align*}
    for some absolute constant $A$. Applying Proposition \ref{prop.chavdarov}, we have 
    \[\mathrm{Prob}(X_{n,D}) \leq \mathrm{Prob} \{u \in U(\mathbb{F}_{q^n}) |\lambda(\mathrm{Frob}_{\mathbb{F}_{q^n},u}) \in C \} \leq \prod_{\ell \in S} \left(\frac{1}{2} + \frac{A}{\ell}\right) + \frac{R_S}{\sqrt{q^n}}\]
    where $R_S$ is a constant only depending on $U$. 
    Now fix $\epsilon > 0$, and pick the size of $S$ large enough so that 
    \[\prod_{\ell \in S} \left(\frac{1}{2} + \frac{A}{\ell}\right) < \epsilon/2,\]
    and fix $n_0$ such that 
    \[\frac{R_S}{\sqrt{q^{n_0}}} < \epsilon/2.\]
    Thus, for $n > n_0$ we have $\mathrm{Prob}(X_n) < \epsilon$. This proves the Theorem in the function field case. 

    The number field case follows very similarly, so we sketch the idea. For every coset $b \subseteq \Gal(K'/K)$, the image $\bar{\rho}_{\ell}(\Gamma_b)$ is a union of cosets of $\Omega_{\ell}$. Proposition \ref{prop.proportion} gives a proportion of 
    \[\frac{1}{2} + O(\ell^{-1}) \]
    of elements in $C_{D,\ell}$, uniformly on each coset. Thus the proportion is the same on a union of them. Applying Proposition \ref{prop.Zywina} and averaging over conjugacy classes in $\Gal(K'/K)$, we have 
    \[\delta(T_S) \leq \prod_{\ell \in S} \left(\frac{1}{2} + \frac{A}{\ell} \right)\] 
    which goes to 0 as the size of $S$ grows. 
\end{proof}

\subsection{Further generalizations}
Lemma \ref{lemma.meyer} indicates that in general, expecting a K3 surface over a global field to have a density 0 set of reductions having a line bundle satisfying our specific norm condition is too crude. Instead, we must take into account the size of forced Picard rank jumps in each monodromy component. To illustrate this problem further, even when global representability is not a problem, local representability poses an issue. 

\begin{rem}
    For the characteristic polynomial of Frobenius to be able to detect a specific special cycle, we pass to the Kuga-Satake abelian scheme. While the Kuga-Satake representation is some power of spin representations, we may restrict to an appropriate spin module $S$ and look for a condition on the characteristic polynomial there. 

    Suppose that $p$ is a point of good reduction such that $\dim W_p \geq 2$ and that $c$ is represented by a vector $v$. Via Tate, some power of Frobenius $Frob_p^N$ fixes $W_p$. Since $v$ represents $c$, its image in $W_p \otimes \mathbb{Q}_{\ell}$ also represents $c$. By definition of the Clifford relation
    \[L_{v}^2 = Q(v) = c,\] 
    where $L_{v}$ is Clifford left multiplication by $v$. Thus after extending scalars to $\mathbb{Q}_{\ell}(\sqrt{c})$, $L_{v}$ is diagonalizable with eigenvalues $\pm \sqrt{c}$. A corresponding lift of $Frob_p^N$ preserves these eigenspaces over $\mathbb{Q}_{\ell}(\sqrt{c})$, and thus its characteristic polynomial must split over $\mathbb{Q}_{\ell}(\sqrt{c})$. 

    The problem with checking this condition is that for most $\ell$, the reduction of our form $Q$ mod $\ell$ is nondegenerate and a nondegenerate form of rank $n \geq 2$ over $\mathbb{F}_{\ell}$ is universal. Thus, via Hensel's lemma, we may lift our solutions to mod $\ell^n$ for all $n$. In other words, naive polynomial factorization or representability of a square class mod $\ell$ is too common a condition for a sieve to work. 
\end{rem}

Thus, a finer sieve condition is required on a monodromy component with generic rank jump of 2 or higher, which we leave for future work.

\bibliographystyle{alpha} 
\bibliography{references}

\end{document}